\documentclass[12pt,a4paper,draft,twoside,reqno]{amsart}
\usepackage{amssymb}
\usepackage{amsfonts}
\usepackage{amsmath}
\newtheorem{theorem}{Theorem}

\newtheorem{corollary}[theorem]{Corollary}

\newtheorem{example}[theorem]{Example}

\newtheorem{lemma}[theorem]{Lemma}

\newtheorem{remark}[theorem]{Remark}

\begin{document}

\title[On equivalence of 3-order linear differential operators]{On equivalence of third order linear differential operators on
two-dimensional manifolds}
\author[Valentin Lychagin]{Valentin Lychagin}
\address{University of Tromso, Tromso, Norway; Institute of Control Sciences of RAS, Moscow, Russia}
\email{Valentin.Lychagin@matnat.uit.no}
\author[Valeriy Yumaguzhin]{Valeriy Yumaguzhin}
\address{
Program Systems Institute of RAS, Pereslavl'-Zales\-skiy, Russia;  Institute of Control Sciences of RAS, Moscow, Russia}
\email{yuma@diffiety.botik.ru}
\thanks{}
%

\begin{abstract}
We study differential invariants of the third order linear differential
operators and use them to find conditions for equivalence of differential
operators acting in line bundles on two dimensional manifolds with respect
to groups of authomorphisms.
\end{abstract}

\maketitle

\section{Introduction}

This paper is a continuation of (\cite{LY2}) and here we analyze equivalence
of the third order linear differential operators, acting in sections of line
bundles over an oriented 2-dimensional manifold with respect to the group of
automorphisms of the line bundle. The method we applied is very similar to
the method used in (\cite{LY2}) and the novel idea we used is to get global
equivalence by using \textit{natural coordinates }delivering by differential
invariants. Also we use Chern (\cite{Ak},\cite{Ch},\cite{Nag}) and Wagner (%
\cite{Wag}) connections instead of Levi-Civita connection in (\cite{LY2})
and show that any regular third order linear differential operators acting
in sections of line bundles defines natural affine connections in the
bundles. This allows us, by using the quantization, to substitute the
operator by its total symbol and then find its field of rational
differential invariants.

A short version of this paper was presented on Conference dedicated to the
70th birthday of Joseph Krasil'shchik, Trieste 2018, and we wish to him the
very long and the very productive life.

\section{Differential operators}

\subsection{Notations}

The following notations we use in this paper. We denote by $\tau
:TM\rightarrow M$ and $\tau ^{\ast }:T^{\ast }M\rightarrow M$ the tangent
and respectively cotangent bundles for a manifold $M.$ By $\mathbf{1}:%
\mathbb{R}\times M\rightarrow M$ we denote the trivial line bundle.

The symmetric and exterior powers of a vector bundle $\pi :E\left( \pi
\right) \rightarrow M$ \ we'll denote by $\mathbf{S}^{k}\left( \pi \right) $
and $\mathbf{\Lambda }^{k}\left(\pi\right) .$ The module of smooth sections
of bundle $\pi $ we denote by $C^{\infty }\left( \pi \right) ,$ and for the
cases tangent, cotangent and the trivial bundles we'll use the following
notations: $\Sigma _{k}\left( M\right) =C^{\infty }\left( \mathbf{S}%
^{k}\left( \tau \right) \right)$ is the module of symmetric $k$-vectors (or
symbols), $\Sigma ^{k}\left( M\right) =C^{\infty }\left( \mathbf{S}%
^{k}\left( \tau ^{\ast }\right) \right)$ is the module of symmetric $k$%
-forms, $\Omega _{k}\left( M\right) =C^{\infty }\left( \mathbf{\Lambda }%
^{k}\left( \tau \right) \right)$ is the module of skew-symmetric $k$-vectors,
$\Omega ^{k}\left( M\right) =C^{\infty }\left( \mathbf{\Lambda }%
^{k}\left( \tau ^{\ast }\right) \right)$ is the module of exterior $k$-forms, and
$C^{\infty}(\mathbf{1})=C^{\infty}(M)$.

The bundles of $k$-jets of sections of bundle $\pi $ we denote by $\pi _{k}:%
\mathbf{J}^{k}\left( \pi \right) \rightarrow M$ and by $\pi _{k,l}:\mathbf{J}%
^{k}\left( \pi \right) \rightarrow \mathbf{J}^{l}\left( \pi \right) $ we
denote the reduction of $k$-jets to $l$-jets, $k\geq l.$

There is the following exact sequence of vector bundles%
\begin{equation*}
0\rightarrow \mathbf{S}^{k}\left( \tau ^{\ast }\right) \otimes\pi 
\rightarrow \mathbf{J}^{k}\left(\pi \right)\stackrel{\pi _{k,k-1}}{%
\longrightarrow}\mathbf{J}^{k-1}\left( \pi \right) \rightarrow 0,
\end{equation*}%
connecting $\left( k-1\right) $ and $k$ -jets.

Differential operators of order $k$ acting from a vector bundle $\alpha $ to
another bundle, say $\beta$, $\Delta :C^{\infty }\left( \alpha \right)
\rightarrow C^{\infty }\left( \beta \right) ,$ could be lifted to operators $%
\widehat{\Delta }:C^{\infty }\left( \widehat{\alpha }\right) \rightarrow
C^{\infty }\left( \widehat{\beta }\right) ,$ where $\widehat{\alpha }$ and $%
\widehat{\beta }$ are vector bundles over $\mathbf{J}^{\infty }\left( \pi
\right) ,$ induced by the projection $\pi _{\infty }:\mathbf{J}^{\infty
}\left( \pi \right) \rightarrow M,$ and operator $\widehat{\Delta }$ is 
defined by the universal property 
\begin{equation*}
j_{k+l}\left( h\right) ^{\ast }\circ \widehat{\Delta }=\Delta \circ
j_{l}\left( h\right) ^{\ast },
\end{equation*}%
for all sections $h\in C^{\infty }\left( \pi \right) $ and $l=0,1,\ldots$ .

We call operators $\widehat{\Delta }$ as \textit{total lift} of $\Delta .$
Here we'll especially use the following two cases. The de Rham operator $%
\Delta =d:\Omega ^{i}\left( M\right) \rightarrow \Omega ^{i+1}\left(
M\right) .$ Then sections $C^{\infty }\left( \widehat{\Lambda^{i}(\tau
^{\ast })}\right) =\Omega _{h}^{i}\left( \mathbf{J}^{\infty }\left( \pi
\right) \right) $ are horizontal differential $i$-forms on $\mathbf{J}%
^{\infty }\left( \pi \right) $ and $\widehat{d}$ the \textit{total
differential}.

In the another case, when $\Delta =d_{\nabla }:C^{\infty }\left( \alpha
\right) \rightarrow C^{\infty }\left( \alpha \right) \otimes \Omega
^{1}\left( M\right) $ is a covariant differential of a connection $\nabla $
in the bundle $\alpha ,$ operator $\widehat{d}_{\nabla }:C^{\infty }\left( 
\widehat{\alpha }\right) \rightarrow C^{\infty }\left( \widehat{\alpha }%
\right) \otimes \Omega _{h}^{1}\left( \mathbf{J}^{\infty }\left( \pi \right)
\right) $ is the \textit{total covariant differential}.

Linear combinations of total lifts of vector fields are called \textit{total
derivations} and linear combinations of compositions of total derivatives
are \textit{total differential operators}.

\subsection{Symbols}

Let $M$ \ be an oriented $2$-dimensional oriented manifold and let $\xi
:E\left( \xi \right) \rightarrow M$ \ be a line bundle. We denote by $%
\mathbf{Diff}_{k}\left( \xi \right) $ the $C^{\infty }\left( M\right) $%
-module of differential operators of order $\leq k,\ $acting in the bundle:$%
\ \ A:C^{\infty }\left( \xi \right) \rightarrow C^{\infty }\left( \xi
\right) ,$ if $A\in \mathbf{Diff}_{k}\left( \xi \right) .$

By \textit{leading or principal symbol} $\sigma _{k}\left( A\right) $ of
operator $A\in \mathbf{Diff}_{k}\left( \xi \right) $ we mean the equivalence
class 
\begin{equation*}
\sigma _{k,A}=A\,\rm{mod}\,\mathbf{Diff}_{k-1}\left( \xi \right) .
\end{equation*}%
It is known that the symbol could be also viewed as a fibre-wise homogeneous
polynomial of degree $k$ on the cotangent bundle $\tau ^{\ast }:T^{\ast
}M\rightarrow M$ and the following sequence

\begin{equation*}
0\rightarrow \mathbf{Diff}_{k-1}\left( \xi \right) \rightarrow \mathbf{Diff}%
_{k}\left( \xi \right) \overset{\sigma }{\rightarrow }
{\Sigma}_k\rightarrow 0 
\end{equation*}%
is exact.

If $\left( x,y\right) $ are local coordinates on $M$ and $e$ is a local
basic section of $\xi $ , then differential operator $A\in \mathbf{Diff}%
_{3}\left( \xi \right) $ can be represented in the form 
\begin{eqnarray*}
A &=&a_{1}\partial _{x}^{3}+3a_{2}\partial _{x}^{2}\partial
_{y}+3a_{3}\partial _{x}\partial _{y}^{2}+a_{4}\partial _{y}^{3}+ \\
&&b_{1}\partial _{x}^{2}+2b_{2}\partial _{x}\partial _{y}+b_{3}\partial
_{y}^{2}+c_{1}\partial _{x}+c_{2}\partial _{y}+a_{0},
\end{eqnarray*}%
where all coefficients are smooth functions and the action of operator $A$
on a section $s=he,$ where $h$ is a smooth function, equals $A\left(
h\right) e.$

The symbol of operator $A$ is a symmetric $3$-vector 
\begin{equation}
\sigma _{3,A}=a_{1}\partial _{x}^{3}+3a_{2}\partial _{x}^{2}\cdot \partial
_{y}+3a_{3}\partial _{x}\cdot \partial _{y}^{2}+a_{4}\partial _{y}^{3}\in 
\Sigma_3(M)  \label{locals}
\end{equation}%
where dots (and degrees) stand for symmetric product of vector fields.

In canonical coordinates $\left( x,y,p_{x},p_{y}\right) $ on $T^{\ast }M$ this tensor is a cubic polynomial (=Hamiltonian)%
\begin{equation*}
\sigma
_{3,A}=a_{1}p_{x}^{3}+3a_{2}p_{x}^{2}p_{y}+3a_{3}p_{x}p_{y}^{2}+a_{4}p_{y}^{3}
\end{equation*}%
on $T^{\ast }M$.

We say that the symbol $\sigma _{3,A}$ is \textit{regular} (at a point or
domain) if $\sigma _{3,A},$ as a polynomial on $T^{\ast }M,$ has
distinct roots. Denote by 
\begin{equation*}
\Delta \left( \sigma _{3,A}\right) =6a_{1}a_{2}a_{3}a_{4}-4\left(
a_{1}a_{3}^{3}+a_{4}a_{2}^{3}\right) +3a_{2}^{2}a_{3}^{2}-a_{1}^{2}a_{4}^{2}
\end{equation*}%
the polynomial proportional to the discriminant of $\sigma _{3,A}.$

Then polynomial $\sigma _{3,A}$ has three distinct real roots if $\Delta
\left( \sigma _{3,A}\right) >0$ and one real and two complex roots if $%
\Delta \left( \sigma _{3,A}\right) <0.$

We'll say that the differential operator $A$ is \textit{hyperbolic} if $%
\Delta \left( \sigma _{3,A}\right) >0$ and \textit{ultrahyperbolic}, if $%
\Delta \left( \sigma _{3,A}\right) <0.$

Locally, symbol of a hyperbolic operator could be presented as a symmetric
product of pair wise linear independent vector fields: $\sigma =X_{1}\cdot
X_{2}\cdot X_{3},$ and therefore there are local coordinates $\left(
x,y\right) $ such that 
\begin{equation}
\sigma _{3,A}=\left( a\partial _{x}+b\partial _{y}\right) \cdot \partial
_{x}\cdot \partial _{y},  \label{hyper}
\end{equation}%
where $a$ and $b$ are smooth functions and $ab\neq 0.$

For the case of ultrahyperbolic symbols we have $\sigma _{3,A}=X\cdot q,$
where $X$ is a non zero vector and $q$ is a positive symmetric $2$-vector.
Therefore, there are local coordinates $\left( x,y\right) $ such that%
\begin{equation}
\sigma =\left( a\partial _{x}+b\partial _{y}\right) \cdot \left( \partial
_{x}^{2}+\partial _{y}^{2}\right) ,  \label{ultrhyper}
\end{equation}%
where $a$ and $b$ are smooth functions and $a^{2}+b^{2}>0.$

\subsection{Groups actions}

In this paper we study orbits of the 3rd order differential operators with
respect to the following groups:

\begin{enumerate}
\item For the scalar differential operators $A\in \mathbf{Diff}_{3}\left( 
\mathbf{1}\right) ,$ the group is the group $\mathcal{G}\left( M\right) $ of
diffeomorphisms of $M$ with the natural action 
\begin{equation}
\phi _{\ast }:A\longmapsto \phi _{\ast }\circ A\circ \phi _{\ast }^{-1},
\label{diffAct}
\end{equation}%
where $\phi _{\ast }=\phi ^{\ast -1}:C^{\infty }\left( M\right) \rightarrow
C^{\infty }\left( M\right) $ is the induced algebra morphism and $\phi
^{\ast }\left( f\right) =f\circ \phi ,$ $f\in C^{\infty }\left( M\right) .$

\item For general linear bundles $\xi $ and operators $A\in \mathbf{Diff}%
_{3}\left( \xi \right) $ we'll take the group of automorphisms $%
\mathbf{Aut}(\xi )$ with the following actions.

For sections $s\in C^{\infty }\left( \xi \right) $ and elements $\widetilde{%
\phi }\in \mathbf{Aut}(\xi ),$ we define action as follows%
\begin{equation*}
\widetilde{\phi }_{\ast }:s\longmapsto \widetilde{\phi }\circ s\circ \phi
^{-1},
\end{equation*}%
and 
\begin{equation}
\widetilde{\phi }_{\ast }:A\longmapsto \widetilde{\phi }_{\ast }\circ A\circ 
\widetilde{\phi _{\ast }}^{-1}.  \label{AutAct}
\end{equation}
\end{enumerate}

These groups are connected by the sequence 
\begin{equation*}
1\rightarrow \mathcal{F}\left( M\right) \rightarrow \mathbf{Aut}(\xi
)\rightarrow \mathcal{G}\left( M\right) \rightarrow 1,
\end{equation*}%
where $\mathcal{F}\left( M\right) $ is the multiplicative group of smooth
nowhere vanishing functions on $M$ .

Later on we'll get conditions under which a diffeomorphism $\phi \in 
\mathcal{G}\left( M\right) $ has a lift in $\mathbf{Aut}(\xi ).$

Assume now that $e$ is a local base for $\xi $ , i.e.nowhere vanishing
section in a domain. Then for any operator $A\in \mathbf{Diff}_{3}\left( \xi
\right) $ we define a scalar operator $A_{e}\in \mathbf{Diff}_{3}\left( 
\mathbf{1}\right) $ in such a way that 
\begin{equation*}
A\left( fe\right) =A_{e}\left( f\right) e,
\end{equation*}%
for $f\in C^{\infty }\left( M\right) .$

It is easy to see that change of base leads us to the following
transformation of scalar models:%
\begin{equation}
A_{e}\longmapsto A_{\widetilde{e}}=h\circ A\circ h^{-1},  \label{funcAct}
\end{equation}%
if $\widetilde{e}=he,$ and $h\in \mathcal{F}\left( M\right) .$

Therefore, (locally) action (\ref{AutAct}) could be considered as
composition of actions (\ref{diffAct}) and (\ref{funcAct}).

\section{Connections and symbols}

In this section $\sigma \in \Sigma_3(M)$ is a regular symmetric $3$%
-vector.

\subsection{Wagner connection}

The following result due to Wagner, (\cite{Wag}).

\begin{theorem}
For any regular symbol $\sigma \in \Sigma_3(M)$ there exists and
unique connection $\nabla ^{\sigma }$ in the tangent bundle such that 
\begin{equation}
d_{\nabla ^{\sigma }}\left( \sigma \right) =0,  \label{Wconn}
\end{equation}%
where $d_{\nabla ^{\sigma }}:\Sigma _{3}\left( M\right) \rightarrow \Omega
^{1}\left( M\right) \otimes \Sigma _{3}\left( M\right) $ \ is the covariant
differential.
\end{theorem}

\begin{proof}
Let's choose a local coordinates $\left( x,y\right) $ on $M,$ and let $%
\Gamma _{ij}^{k}$ be the Christoffel symbols of connection $\nabla $ that
satisfies the theorem.

Then 
\begin{equation*}
\partial _{l}a_{ijk}+\Gamma _{lm}^{i}a_{mjk}+\Gamma _{lm}^{j}a_{imk}+\Gamma
_{lm}^{k}a_{ijm}=0,
\end{equation*}%
where $a_{ijk}$ are components of $\sigma $ and in the previous notations $%
a_{111}=a_{1},a_{112}=a_{211}=a_{121}=a_{2},$ $a_{122}=a_{212}=a_{221}=a_{3}$
and $a_{222}=a_{4}.$

Consider this as a system of linear equations with respect to Christoffel
symbols $\Gamma _{ij}^{k}.$ This is $8\times 8$ system (we did not assumed
that $\nabla $ is a torsion free connection). The determinant of this system
equals $81\Delta \left( \sigma \right) ^{2}\neq 0,$ therefore (\ref{Wconn})
determines the unique connection.
\end{proof}

We'll call connection defining by (\ref{Wconn}), \textit{Wagner connection}.

In the case of hyperbolic symbols we'll choose local coordinates in such a
way that $\sigma $ has formula (\ref{hyper}) holds. Then the non zero
Christoffel coefficients of the Wagner connection could be found from the
above system of linear equations:%
\begin{eqnarray*}
\Gamma _{11}^{1} &=&\frac{1}{3}\left( \ln \frac{b}{a^{2}}\right) _{x},\ \
\Gamma _{22}^{2}=\frac{1}{3}\left( \ln \frac{a}{b^{2}}\right) _{y}, \\
\Gamma _{12}^{1} &=&\frac{1}{3}\left( \ln \frac{b}{a^{2}}\right) _{y},\ \
\Gamma _{21}^{2}=\frac{1}{3}\left( \ln \frac{a}{b^{2}}\right) _{x}.
\end{eqnarray*}

In the similar way, for the ultrahyperbolic symbols (\ref{ultrhyper}), we
get 
\begin{eqnarray*}
\Gamma _{11}^{1} &=&\Gamma _{21}^{2}=-\frac{1}{6}[\ln \left(
a^{2}+b^{2}\right) ]_{x},\ \ \Gamma _{12}^{1}=\Gamma _{22}^{2}=-\frac{1}{6}%
[\ln \left( a^{2}+b^{2}\right) ]_{y}, \\
\Gamma _{21}^{1} &=&-\Gamma _{11}^{2}=\frac{a_{x}b-ab_{x}}{a^{2}+b^{2}},\ \
\Gamma _{22}^{1}=-\Gamma _{12}^{2}=\frac{ab_{y}-a_{y}b}{a^{2}+b^{2}}.
\end{eqnarray*}

\begin{corollary}
\begin{enumerate}
\item The curvature of the Wagner connection equals zero.

\item The torsion form of the Wagner connection equals%
\begin{equation*}
\theta _{\sigma }^{h}=\frac{1}{3}\left( \ln \frac{b^{2}}{a}\right) _{x}dx+%
\frac{1}{3}\left( \ln \frac{a^{2}}{b}\right) _{y}dy,
\end{equation*}%
for the hyperbolic case, and%
\begin{equation*}
\theta _{\sigma }^{u}=\frac{ab_{y}-a_{y}b}{a^{2}+b^{2}}dx+\frac{a_{x}b-ab_{x}%
}{a^{2}+b^{2}}dy-\frac{1}{6}d[\ln \left( a^{2}+b^{2}\right) ],
\end{equation*}%
for the ultrahyperbolic case.
\end{enumerate}
\end{corollary}

\subsection{Wagner metric}

Denote by $\Sigma_{0}^{k}\left( M\right) \subset\Sigma%
^{k}\left( M\right) $ and $\Sigma_{k,0}\left( M\right) \subset 
\Sigma_{k}\left( M\right) $  subsets of \textit{regular }symmetric
differential $k$-forms and \textit{regular} symbols, respectively, i.e.
symmetric tensors with non zero discriminant.

The following theorem holds for two-dimensional manifolds and it is also due
to Wagner (\cite{Wag}).

\begin{theorem}
There is a natural mapping%
\begin{equation*}
W:\Sigma_{3,0}\left( M\right) \rightarrow \Sigma_{0}^{2}\left( M\right) ,
\end{equation*}%
i.e. a mapping commuting with the action of the diffeomorphism group.
\end{theorem}

\begin{proof}
Let's take a point $a\in M$ \ and let $V=T_{a}^{\ast }\ $be the cotangent
space and $\sigma $ be a regular symbol, which we'll consider as homogeneous
cubic polynomial on $V.$ Denote by $\rm{Hess}\left( f\right) $ the
Hessian of a function $f$ on $V,$ computing in a fixed coordinates on the
vsector space.

Then, $\rm{Hess}\left( \lambda f\right) =\lambda ^{2}\rm{Hess}%
\left( f\right) ,$ \ for $\lambda \in \mathbb{R}$ , and 
\begin{equation*}
\rm{Hess}\left( A^{\ast }(f)\right) =\left( \det A\right) ^{2}A^{\ast
}\left( \rm{Hess}\left( f\right) \right) ,
\end{equation*}%
for any linear map $A:V\rightarrow V.$

Remark, that $\rm{Hess}\left( f\right) $ is a quadratic function on $V,$
when $f$ is a cubic. Therefore,%
\begin{equation*}
\rm{Hess}_{2}\left( \sigma \right) =\rm{Hess}\left( 
\rm{Hess}\left( \sigma \right) \right) \in \mathbb{R},
\end{equation*}%
is a scalar.

The straightforward computations show that $\rm{Hess}%
_{2}\left( \sigma \right) $ proportional to the discriminant of $%
\sigma $ and, therefore $\rm{Hess}_{2}\left( \sigma \right)
\neq 0,$ if $\sigma $ regular tensor.

On the other hand, we have 
\begin{eqnarray*}
&&\rm{Hess}_{2}\left( A^{\ast }\left( \sigma \right) \right) =%
\rm{Hess}\left( \rm{Hess}\left( A^{\ast }\left( \sigma \right)
\right) \right) =\rm{Hess}\left( \det \left( A\right) ^{2}\ A^{\ast
}\left( \rm{Hess}\left( \sigma \right) \right) \right) = \\
&&\det \left( A\right) ^{4}\rm{Hess}\left( \ A^{\ast }\left( \rm{%
Hess}\left( \sigma \right) \right) \right) =\det \left( A\right) ^{6}%
\rm{Hess}_{2}\left( \sigma \right) .
\end{eqnarray*}

Put 
\begin{equation}
g_{k}\left( \sigma \right) =\rm{Hess}_{2}\left( \sigma \right)
^{k}\rm{Hess}\left( \sigma \right) ,  \label{gk}
\end{equation}%
for $k\in \mathbb{R}.$

Then, 
\begin{equation*}
g_{k}\left( A^{\ast }\left( \sigma \right) \right) =\rm{Hess}%
_{2}\left( A^{\ast }\left( \sigma \right) \right) ^{k}\rm{Hess}%
\left( A^{\ast }\left( \sigma \right) \right) =\det \left( A\right)
^{6k+2}A^{\ast }\left( g_{k}\left( \sigma \right) \right) .
\end{equation*}

Therefore, quadratic form $g_{-1/3}$ behave in the natural way with respect
to linear transformations:%
\begin{equation*}
A^{\ast }\left( g_{-1/3}\left( \sigma \right) \right) =g_{-1/3}\left(
A^{\ast }\left( \sigma \right) \right) ,
\end{equation*}%
and the mapping 
\begin{equation*}
W\left( \sigma \right) =\sqrt[3]{\rm{Hess}_{2}\left( \sigma
\right) }\rm{Hess}\left( \sigma \right) ^{-1},
\end{equation*}%
where we denoted by $\rm{Hess}\left( \sigma \right) ^{-1}\in S^{2}V$
the metric inverse to $\rm{Hess}\left( \sigma \right) \in S^{2}V^{\ast
},$ satisfies the conditions of the theorem.
\end{proof}

\begin{remark}
If tensor $\sigma \in\Sigma_{0,3}\left( M\right) $ in local
coordinates has form (\ref{locals}), then: 
\begin{eqnarray*}
\rm{Hess}\left( \sigma \right) &=&\left( a_{1}a_{3}-a_{2}^{2}\right)
\partial _{x}^{2}+\left( a_{1}a_{4}-a_{3}a_{2}\right) \partial _{x}\cdot
\partial _{y}+\left( a_{2}a_{4}-a_{3}^{2}\right) \partial _{y}^{2}, \\
\rm{Hess}_{2}\left( \sigma \right) &=&\Delta \left( \sigma
\right) =6a_{1}a_{2}a_{3}a_{4}-4\left( a_{1}a_{3}^{2}+a_{2}^{2}a_{4}\right)
+3a_{2}^{2}a_{3}^{2}-a_{1}^{2}a_{4}^{2},
\end{eqnarray*}%
and 
\begin{multline*}
W\left( \sigma \right) =\frac{4}{\Delta \left( \sigma \right) ^{2/3}}\big(
( a_{2}a_{4}-a_{3}^{2})\,dx^{2}+(a_{2}a_{3}-a_{1}a_{4})\, dx\cdot dy
\\+( a_{1}a_{3}-a_{2}^{2})\,dy^{2}\big).
\end{multline*}%
Remark that the type of this metric depends on sign of $\Delta \left( \sigma
\right) .$ Namely, $W\left( \sigma \right) $ is the definite metric in the
hyperbolic case and indefinite for the ultrahyperbolic one.
\end{remark}

\begin{remark}
In the similar way one might show that there is also natural map $
\Sigma_{0}^{3}\left( M\right) \rightarrow\Sigma_{0}^{2}\left(
M\right) .$
\end{remark}

\subsection{Group-type symbols}

As the first application of the Wagner connection we'll analyze the case
when a symbol $\sigma \in\Sigma_{3,0}\left( M\right) $ has a
group nature, i.e. there is a hidden group Lie that acts in a transitive way
on the manifold, and the symbol is a group invariant.

It is known that if on a manifold there is a connection $\nabla $ with zero
curvature then the torsion tensor $T_{\nabla }$ defines a skew symmetric
bracket on vector fields. If we consider only covariant constant vector
fields and a assume that torsion tensor also covariant constant, the bracket
given by $T_{\nabla }$ convert the vector space of covariant constant vector
fields into a Lie algebra. Moreover, this construction shows that parallel
transports with respect to $\nabla $ are exactly left multiplications in the
corresponding (local) Lie group.

Applying these remarks to the Wagner connections $\nabla ^{\sigma },$
corresponding to the regular symbols $\sigma $ we get the following result.

\begin{theorem}
Let $\sigma \in\Sigma_{3,0}\left( M\right) $ be a regular symbol
and let $\nabla ^{\sigma }$ be the corresponding Wagner connection. Then:

\begin{enumerate}
\item Symbol $\sigma $ is locally equivalent to the symbol with constant
coefficients%
\begin{equation*}
\sigma
=c_{1}p_{1}^{3}+3c_{2}p_{1}^{2}p_{2}+3c_{3}p_{1}p_{2}^{2}+c_{4}p_{2}^{3},\ \
c_{i}\in \mathbb{R},
\end{equation*}%
if and only if $T_{\nabla ^{\sigma }}=0.$

\item Symbol $\sigma $ is locally equivalent to the symbol of the form%
\begin{equation*}
\sigma =c_{1}\exp \left( 3y\right) p_{1}^{3}+3c_{2}\exp \left( 2y\right)
p_{1}^{2}p_{2}+3c_{3}\exp \left( y\right) p_{1}p_{2}^{2}+c_{4}p_{2}^{3},\ \
c_{i}\in \mathbb{R},
\end{equation*}%
if and only if $T_{\nabla ^{\sigma }}\neq 0,$ but%
\begin{equation*}
d_{\nabla ^{\sigma }}\left( T_{\nabla ^{\sigma }}\right) =0.
\end{equation*}
\end{enumerate}
\end{theorem}

\begin{remark}
\bigskip In the second case the discussed above Lie algebra $\mathfrak{g\ }$%
is the solvable $2$-dimensional Lie algebra generated by vector fields%
\begin{equation*}
\mathfrak{g=}\left\langle \partial _{x},\partial _{y}+x\partial
_{x}\right\rangle .
\end{equation*}
\end{remark}

\subsection{Chern connection}

Together with the Wagner connection we'll consider also another connection,
which we call \textit{Chern connection. }This connection\textit{\ }is
similar to the connection used by Chern in geometry of plane 3-webs (see,
for example,\cite{Ak}, \cite{Ch} ,\cite{Nag}).

Namely, by Chern connection, associated with symbol $\sigma ,$ we'll
understand a linear connection $\nabla ^{c}$ on manifold $M$, that preserves
conformal classes of $\sigma :$%
\begin{equation}
d_{\nabla ^{c}}\left( \sigma \right) =\omega \otimes \sigma ,
\label{Chern connection}
\end{equation}%
for some differential form $\omega \in \Omega ^{1}\left( M\right) .$

\begin{theorem}
For any regular symmetric 3-vector $\sigma \in \Sigma _{3}$ there exist and
unique torsion free Chern connection, and 
\begin{equation}
d_{\nabla ^{\sigma }}-d_{\nabla ^{c}}=\alpha \otimes \rm{Id},
\label{Chern-Wagner}
\end{equation}%
where $\alpha =\theta _{\sigma }$ is the torsion form of the Wagner
connection, and $\omega =-3\theta _{\sigma }.$

Curvature of the Chern connection equals $d\omega .$
\end{theorem}

\begin{proof}
Consider (\ref{Chern connection}) as a linear system of equations with
respect to components of symmetric connection $\nabla ^{c}$ and differential
form $\omega .$ It gives us a $8\times 8$ system of linear equations with
determinant $-9\Delta \left( \sigma \right) ^{2}\neq 0.$

On the other hand we have relation (\ref{Chern-Wagner}) because of
stationary Lie algebra of a regular symmetric 3-vector is trivial, and
therefore $\theta _{\sigma }=\alpha .$
\end{proof}

\begin{example}
The Chern curvature of hyperbolic symbol (\ref{hyper}) equals to 
\begin{equation*}
\ln \left( \frac{a}{b}\right) _{xy}dx\wedge dy.
\end{equation*}%
In the case of ultrahyperbolic symbol we can choose a representative in the
conformal class of symbol (\ref{ultrhyper}) with $a=\sin \left( \phi \left(
x,y\right) \right) $ and $b=\cos \left( \phi \left( x,y\right) \right) .$
Then the Chern curvature equals 
\begin{equation*}
\left( \phi _{xx}+\phi _{yy}\right) dx\wedge dy.
\end{equation*}
\end{example}

\begin{theorem}
Conformal classes of regular symbols are locally equivalent to conformal
classes of symbols with constant coefficients if and only if the curvature
of their Chern connection vanishes.
\end{theorem}

\begin{proof}
The statement easily to check for the hyperbolic case. \ In the
ultrahyperbolic case it can be seen in complex coordinates $z=x+\sqrt{-1}y,%
\overline{z}=x-\sqrt{-1}y.$
\end{proof}

\section{Classification of regular symmetric 3-vectors}

Let $\pi :\mathbf{S}^{3}\tau \left( M\right) \rightarrow M$ be the vector
bundle of symmetric $3$-vectors, $C^{\infty }\left( \pi \right) =\Sigma
_{3}\left( M\right) ,$ and let $\pi _{k}:\mathbf{J}^{k}\left( \pi \right)
\rightarrow M$ be the bundles of $k$-jets of symmetric $3$-vectors.

The group of diffeomorphisms $\mathcal{G}\left( M\right) $ acts in the
natural way in the bundle $\pi $ as well as (by prolongations) in jet
bundles $\pi _{k}.$

In this section we study orbits of this actions. First of all, because the
action of $\mathcal{G}\left( M\right) $ on $M$ transitive, we could fix a
point $a\in M$ and restrict ourselves by actions of subgroup $\mathcal{G}%
_{a}\left( M\right) =\left\{ \left. \phi \in \mathcal{G}\left( M\right)
\right\vert \phi \left( a\right) =a\right\} $ on fibres $\mathbf{J}%
_{a}^{k}=\pi _{k}^{-1}\left( a\right) .$ Secondly, for the case $k=0,$ we
have two open orbits $\mathcal{O}_{h}$ and $\mathcal{O}_{u}$ that correspond
to the hyperbolic and ultrahyperbolic symbols. The complement to them
consists of symbols having multiple roots.

\subsection{Invariant coframe}

Consider now the orbits in $1$-jets space. As we have seen, the torsion
forms $\theta _{\sigma }$, for regular symbols $\sigma \in \Sigma
_{3,0}\left( M\right) ,$ depend on $1$-jet of $\sigma $ and therefore
defines a horizontal differential $1$-form $\Theta $ on $\pi
_{1,0}^{-1}\left( \mathcal{O}_{h}\cup \mathcal{O}_{u}\right) $ $\subset 
\mathbf{J}^{1}\left( \pi \right) ,$ such that%
\begin{equation*}
j_{1}\left( \sigma \right) ^{\ast }\left( \Theta \right) =\theta _{\sigma },
\end{equation*}%
in the domain where $\sigma $ is regular.

Here we denoted by $j_{1}\left( \sigma \right) :M\rightarrow \mathbf{J}%
^{1}\left( \pi \right) $ the section that corresponds to $1$-jet of $\sigma
. $

In the similar way the Wagner metric $W\left( \sigma \right) \in \Sigma
^{2}\left( M\right) $ defines a quadratic horizontal quadratic form $\mathbf{%
W}$ on $\pi _{1,0}^{-1}\left( \mathcal{O}_{h}\cup \mathcal{O}_{u}\right) $ $%
\subset \mathbf{J}^{1}\left( \pi \right) $ in such a way, that 
\begin{equation*}
j_{1}\left( \sigma \right) ^{\ast }\left( \mathbf{W}\right) =W\left( \sigma
\right) ,
\end{equation*}%
in the domain where $\sigma $ is regular.

Both tensors $\Theta $ and $\mathbf{W}$ are $\mathcal{G}\left( M\right) $%
-invariants. Using orientation on $M$ we construct an oriented $\mathcal{G}%
\left( M\right) $ invariant coframe $\left\langle \Theta ,\Theta ^{\prime
}\right\rangle $ on $\pi _{1,0}^{-1}\left( \mathcal{O}_{h}\cup \mathcal{O}%
_{u}\right) $, where $\Theta ^{\prime }$ is a horizontal form such that 
\begin{equation*}
\mathbf{W}\left( \Theta ,\Theta ^{\prime }\right) =0,\mathbf{W}\left( \Theta
,\Theta \right) =\mathbf{W}\left( \Theta ^{\prime },\Theta ^{\prime }\right)
.
\end{equation*}%
In order to get form $\Theta ^{\prime }$ we need that 
\begin{equation}
\mathbf{W}\left( \Theta ,\Theta \right) \neq 0.  \label{regTheta}
\end{equation}

Therefore, the above coframe exist in domain $\pi _{1,0}^{-1}\left( \mathcal{%
O}_{h}\cup \mathcal{O}_{u}\right) ,$ where (\ref{regTheta}) holds. We denote
the last domain by $\mathcal{O}^{\left( 1\right) }\subset \mathbf{J}%
^{1}\left( \pi \right) $ and call symbols $\sigma $ $1$\textit{-regular }if
its $1$-jet belong to $\mathcal{O}^{\left( 1\right) }.$

\subsection{Universal symbol}

Let's denote by $\widehat{X}$ the total derivation, ~$\widehat{X}:C^{\infty
}\left( \mathbf{J}^{k}\left( \pi \right) \right) \rightarrow C^{\infty
}\left( \mathbf{J}^{k+1}\left( \pi \right) \right) ,$ that corresponds to
vector field $X$ on manifold $M,$ (see, \cite{KLV}), and let $\mathcal{D}%
_{H}(\pi )$ be the module over $C^{\infty }\left( \mathbf{J}^{\infty }\left(
\pi \right) \right) ,$ generated by the total derivatives. Elements of this
module we call \textit{horizontal fields}.

Elements of the symmetric cube of this module, i.e. $S^{3}\left( \mathcal{D}%
_{H}(\pi )\right) ,$ we'll call \textit{horizontal symbols.}

The basic property of horizontal fields consist of the fact that they
preserve and tangent to the Cartan distribution and therefore they could be
restricted on sections of $k$-jet bundles $\pi _{k}$ of the form $%
j_{k}\left( \sigma \right) $ $.$

The proof of the following theorem is standard for universal constructions (%
\cite{KLV},\cite{LY2}).

\begin{theorem}
There exists and unique an universal horizontal symbol $\Xi _{3}$ such that
the restriction of this symbol on $j_{0}\left( \sigma \right) $ coincides
with $\sigma .$ \newline
The symbol is an invariant of the diffeomorphism group.
\end{theorem}

\subsubsection{Coordinates}

Denote by $\left( x,y\right) $ local coordinates on $M$ and lets $\left(
x,y,u^{1},u^{2},u^{3},u^{3}\right) $ be the corresponding standard local
coordinates in $S^{3}\left( T\right) \left( M\right) ,$ Then the
section $j_{0}\left( \sigma \right) ,$ corresponding to symbol,%
\begin{equation*}
\sigma =a_{1}\left( x,y\right) \partial _{x}^{3}+3a_{2}\left( x,y\right)
\partial _{x}^{2}\cdot \partial _{y}+3a_{3}\left( x,y\right) \partial
_{x}\cdot \partial _{y}^{2}+a_{4}\left( x,y\right) \partial _{y}^{3},
\end{equation*}%
has the form%
\begin{equation*}
u^{1}=a_{1}\left( x,y\right) ,u^{2}=a_{2}\left( x,y\right)
,u^{3}=a_{3}\left( x,y\right) ,u^{4}=a_{4}\left( x,y\right) .
\end{equation*}

It is easy to check, that the universal symbol has the form:%
\begin{equation*}
\Xi _{3}=u^{1}\left( \frac{d}{dx}\right) ^{3}+3u^{2}\left( \frac{d}{dx}%
\right) ^{2}\cdot \left( \frac{d}{dy}\right) +3u^{3}\left( \frac{d}{dx}%
\right) \cdot \left( \frac{d}{dy}\right) ^{2}+u^{4}\left( \frac{d}{dy}%
\right) ^{3},
\end{equation*}%
in these coordinates.

\subsection{Differential invariants of symbols}

\subsubsection{Symbols}

First of all, as we have seen, there are no none constant functions on $%
\mathbf{J}^{0}\left( \pi \right) ,$ which are invariant with respect to the
diffeomorphism pseudo group. Indeed, in this case we have to open orbits $%
\mathcal{O}_{h}\ $and$\ \mathcal{O}_{u}$ such that the closure of their
union coincide with $\mathbf{J}^{0}\left( \pi \right) .$

The action of the diffeomorphism pseudo group into $1$-jet space $\mathbf{J}%
^{1}\left( \pi \right) $ has invariant coframe $\left\langle \Theta ,\Theta
^{\prime }\right\rangle $ and invariant universal symbol.

Let horizontal vector fields $\delta _{1}$ and $\delta _{2}$ constitute a
frame dual to coframe $\left\langle \Theta ,\Theta ^{\prime }\right\rangle $
and let 
\begin{equation*}
\Xi _{3}=I_{1}\delta _{1}^{3}+3I_{2}\delta _{1}^{2}\cdot \delta
_{2}+3I_{3}\delta _{1}\cdot \delta _{2}^{2}+I_{4}\delta _{2}^{3},
\end{equation*}%
be the decomposition of the universal symbol in this frame.

Then functions $I_{i}$, $i=1,..,4,$ are defined on the domain $\mathcal{O}%
^{\left( 1\right) }\subset \mathbf{J}^{1}\left( \pi \right) $ and they are \
invariants with respect to the diffeomorphism group action.

Their structure we can described in the following way. Let's denote by $%
\mathcal{F}_{1}$ the field of rational functions on the fibre $\mathbf{J}%
_{a}^{1}\left( \pi \right) ,$ for some fixed point $a\in M,$ and let $%
\mathcal{F}_{1}\left( \sqrt[3]{\Delta }\right) $ be the field extension of $%
\mathcal{F}_{1}$ by $\sqrt[3]{\Delta },$ where $\Delta \left( j_{1}\left(
\sigma \right) \right) =\Delta \left( \sigma \right) .$ Then the
restrictions of functions $I_{i}$ on fibres $\mathbf{J}_{b}^{1}\left( \pi
\right) $ do no depend on $b\in M$ and therefore its enough to consider
their restriction on the fibre $\mathbf{J}_{a}^{1}\left( \pi \right) .$ On
the other hand, their construction shows, that $I_{i}\in \mathcal{F}%
_{1}\left( \sqrt[3]{\Delta }\right) .$

Moreover, the horizontal vector fields $\delta _{1},\delta _{2}$ are also
invariants and therefore their action on invariants give us new invariants.

Let's call by \textit{natural differential invariants of symbols }functions $%
I$ defined on $\mathbf{J}^{k}\left( \pi \right) $ which are invariants with
respect to the diffeomorphism group action and which belong to fields $%
\mathcal{F}_{k}\left( \sqrt[3]{\Delta }\right) ,$ where $\mathcal{F}_{k}$
are the fields of rational functions on fibres $\mathbf{J}_{b}^{k}\left( \pi
\right) ,$ for $b\in M.$ The number $k$ is the order of the invariant.

Summing up this we arrive at the following.

\begin{theorem} $\phantom{\alpha}$\\
\begin{enumerate}
\item The field of natural differential invariants of symbols is generated
by basic invariants $I_{1},I_{2},I_{3},I_{4}$ and the Tresse derivations $%
\delta _{1},\delta _{2}.$
\item This field separate regular orbits.
\end{enumerate}
\end{theorem}

\subsubsection{Conformal classes of symbols}

Let $[\pi ]:\mathbb{P}\left( \mathbf{S}^{3}\tau \right) \left( M\right)
\rightarrow M$ be the projectivization of the symbol bundle $\pi :\mathbf{S}%
^{3}\tau \left( M\right) \rightarrow M.$

Sections of this bundle we'll consider as \textit{conformal classes of
symbols }and denote by $[\sigma ]=\left\{ f\sigma ,\ f\in \mathcal{F}\left(
M\right) \right\} $ the conformal class of a symbol $\sigma .$

Let $[\pi ]_{k}:\mathbf{J}^{k}\left( [\pi ]\right) \rightarrow M$ be the
bundle of $k$-jets of conformal classes. The actions of group $\mathcal{G}%
\left( M\right) $ of diffeomorphisms in the bundles $\pi _{k}$ induce the
actions in bundles $[\pi ]_{k}$ and in this section we study orbits and
invariants of these actions.

First of all, it is clear that this action in $\mathbf{J}^{0}\left( [\pi
]\right) $ has two open orbits $\mathcal{O}_{h}$ and $\mathcal{O}_{u},$
where $\Delta \neq 0,$ and singular orbits that correspond to symbols with
multiple roots.

The Chern connection $\nabla ^{\lbrack \sigma ]}$ depends on conformal class
of the regular symbol $\sigma ,$ and 
\begin{equation*}
d_{\nabla ^{_{\lbrack \sigma ]}}}\left( \sigma \right) =\omega _{\sigma
}\otimes \sigma ,
\end{equation*}%
where differential $1$-form $\omega _{\sigma }$ depends on representative $%
\sigma $ in the conformal class $[\sigma ]$ in the following way:%
\begin{equation*}
\omega _{f\sigma }=\omega _{\sigma }+d\ln \left\vert f\right\vert ,
\end{equation*}%
where $f\in \mathcal{F}\left( M\right) .$

The curvature of the Chern connection $\nabla ^{_{\lbrack \sigma ]}},$ as a
tensor in $\Omega ^{2}\left( M\right) \otimes \rm{End}\left( \tau
\right) $ is the scalar operator $d\omega _{\sigma }\otimes \rm{Id}.$

As we have seen, the Christoffel symbols of the Chern connection $\nabla
^{\lbrack \sigma ]}$ depend on the first order jets of $\sigma $ an
therefore the curvature 2-form $d\omega _{\sigma }$ depends on the second
jets.

As above, let's denote by $\Omega $ an \textit{universal horizontal 2-form}
on the manifold of second jets $\mathcal{O}^{\left( 2\right) }=[\pi
]_{2,0}^{-1}\left( \mathcal{O}\right) ,$ where $\mathcal{O}=\mathcal{O}%
_{h}\cup \mathcal{O}_{u},$ such that restrictions of $\Omega $ on $2$-jet
sections $j_{2}\left( [\sigma ]\right) :M\rightarrow \mathbf{J}^{2}\left(
[\pi ]\right) $ coincide with curvature form $d\omega _{\sigma }:$%
\begin{equation*}
j_{2}\left( [\sigma ]\right) ^{\ast }\left( \Omega \right) =d\omega _{\sigma
},
\end{equation*}%
for regular symbols.

Applying the total covariant differential 
$$\widehat{d_{\nabla }}:\Omega
_{h}^{2}\left( \mathbf{J}^{2}\left( [\pi ]\right) \right) \rightarrow \Omega
_{h}^{2}\left( \mathbf{J}^{3}\left( [\pi ]\right) \right) \otimes \Omega
_{h}^{1}\left( \mathbf{J}^{3}\left( [\pi ]\right) \right),
$$ 
$\nabla =\nabla^{\lbrack \sigma ]}$, to the universal 2-form we get 
\begin{equation*}
\widehat{d_{\nabla }}\left( \Omega \right) =\Omega \otimes \theta ,
\end{equation*}%
where $\theta \in \Omega _{h}^{1}\left( \mathbf{J}^{3}\left( [\pi ]\right)
\right) $ is a horizontal 1-form on the space of 3-jets.

Repeating this procedure and applying the total covariant differential $%
\widehat{d_{\nabla }}:\Omega _{h}^{1}\left( \mathbf{J}^{3}\left( [\pi
]\right) \right) \rightarrow \Omega _{h}^{1}\left( \mathbf{J}^{4}\left( [\pi
]\right) \right) \otimes \Omega _{h}^{1}\left( \mathbf{J}^{4}\left( [\pi
]\right) \right) $ to the horizontal 1-form $\theta $ we get tensor $%
\widehat{d_{\nabla }}\left( \theta \right) $ on the space of 4-jets. We take
the symmetric part of this tensor (it is easy to check that the skew
symmetric part is proportional to $\Omega $) and get horizontal quadratic
differential form $G$ on the space of 4-jets.

Denote by $\mathcal{O}^{\left( 4\right) }\subset \lbrack \pi
]_{4,0}^{-1}\left( \mathcal{O}\right) $ the domain of regular 4-jets, i.e.
such 4-jets where $G$ is non degenerated and $G\left( \theta ,\theta \right)
\neq 0,\Omega \neq 0.$ Then, similar to the above, we construct the $G-$%
orthogonal coframe $\left\langle \theta ,\theta ^{\prime }\right\rangle $ in
the domain $\mathcal{O}^{\left( 4\right) }.$ This coframe is invariant of
the diffeomorphism group and decomposing the projective class of the
universal symbol $\Xi $ we get $\mathcal{G}\left( M\right) $-invariant
mapping: 
\begin{equation*}
I:\mathcal{O}^{\left( 4\right) }\rightarrow \lbrack
I_{1}:I_{2}:I_{3}:I_{4}]\in \mathbb{P}^{3},
\end{equation*}%
where functions $I_{i}/I_{l}$ are natural invariants of order 4.

\begin{theorem}$\phantom{\alpha}$\\
The field of natural differential invariants of conformal classes symbols is
generated by basic invariants $I_{i}/I_{l}$ 
and the Tresse derivations $\delta _{1},\delta _{2},$where frame $\left\langle \delta
_{1},\delta _{2}\right\rangle $ is dual to coframe $\left\langle \theta
,\theta ^{\prime }\right\rangle .$This field separates regular orbits.
\end{theorem}

\subsubsection{Coordinates}

In hyperbolic case we'll take such coordinates $\left( x,y\right) $ that the
symbol has the form:%
\begin{equation*}
\sigma =\left( \partial _{x}+\exp \left( h\right) \partial _{y}\right) \cdot
\partial _{x}\cdot \partial _{y},
\end{equation*}%
where $h=h\left( x,y\right) $ is a smooth function.

Then, the non zero Christoffel symbols are 
\begin{equation}
\Gamma _{1,1}^{1}=h_{x},\ \Gamma _{2,2}^{2}=-h_{y},  \label{HyperChern}
\end{equation}%
and the curvature form 
\begin{equation*}
\Omega =-3h_{xy}dx\wedge dy,
\end{equation*}%
The corresponding 1-form equals 
\begin{equation*}
\theta =d\ln \left( h_{xy}\right) -h_{x}dx+h_{y}dy.
\end{equation*}%
The covariant differential of an 1-form 
\begin{equation*}
\alpha =Pdx+Qdy
\end{equation*}%
equals 
\begin{multline*}
d_{\nabla }\left( \alpha \right) =\left( P_{x}-Ph_{x}\right) dx\otimes
dx+P_{y}dx\otimes dy+Q_{x}dy\otimes dx\\+\left( Q_{y}+Qh_{y}\right) dy\otimes
dy,
\end{multline*}%
and therefore 
\begin{equation*}
G=\left( P_{x}-Ph_{x}\right) dx\cdot dx+\left( P_{y}+Q_{x}\right) dx\cdot
dy+\left( Q_{y}+Qh_{y}\right) dy\cdot dy,
\end{equation*}%
where 
\begin{equation*}
P=\frac{h_{xxy}}{h_{xy}}-h_{x},\ Q=\frac{h_{xyy}}{h_{xy}}+h_{y}.
\end{equation*}

In the ultrahyperbolic case we take such coordinates $\left( x,y\right) $
that the symbol has the form%
\begin{equation*}
\sigma =\left( \sin \left( h\right) \partial _{x}+\cos \left( h\right)
\partial _{y}\right) \cdot \left( \partial _{x}^{2}+\partial _{y}^{2}\right)
,
\end{equation*}%
where, as above, $h=h\left( x,y\right) $ is a smooth function.

Then, the non zero Christoffel symbols are 
\begin{equation}
\Gamma _{1,1}^{1}=h_{y},\ \Gamma _{1,2}^{1}=-h_{x},\ \Gamma
_{1,2}^{2}=h_{y},\ \Gamma _{2,2}^{2}=-h_{x},  \label{ultraChern}
\end{equation}%
and the curvature form 
\begin{equation*}
\Omega =-3\left( h_{xx}+h_{yy}\right) dx\wedge dy.
\end{equation*}%
The corresponding 1-form equals%
\begin{equation*}
\theta =d\ln \left( h_{xx}+h_{yy}\right) -2h_{y}dx+2h_{x}dy.
\end{equation*}

The covariant differential of an 1-form $\alpha $ equals 
\begin{eqnarray*}
d_{\nabla }\left( \alpha \right) &=&\left( P_{x}-Ph_{y}-Qh_{x}\right)
dx\otimes dx+\left( P_{y}+Ph_{x}-Qh_{y}\right) dx\otimes dy+ \\
&&\left( Q_{x}+Ph_{x}-Qh_{y}\right) dy\otimes dx+\left(
Q_{y}+Ph_{x}+Qh_{y}\right) dy\otimes dy,
\end{eqnarray*}%
and therefore 
\begin{multline*}
G=\left( P_{x}-Ph_{y}-Qh_{x}\right) dx\cdot dx+\left( P_{y}+Q_{x}+2\left(
Ph_{x}-Qh_{y}\right) \right) dx\cdot dy\\+\left( Q_{y}+Ph_{x}+Qh_{y}\right)
dy\cdot dy,
\end{multline*}%
where%
\begin{equation*}
P=\frac{h_{xxx}+h_{xyy}}{h_{xx}+h_{yy}}-2h_{y},\ Q=\frac{h_{xxy}+h_{yyy}}{%
h_{xx}+h_{yy}}+2h_{x}.
\end{equation*}

\section{Quantization and splitting of scalar differential operators}

\subsection{Quantization}

In this section we apply the quantization procedure outlined in (\cite{LY2}
and \cite{LQ}).

Let $\Sigma ^{\cdot }=\oplus _{k\geq 0}\Sigma ^{k}\left( M\right) $ be the
graded algebra of symmetric differential forms on the surface $M$. Then any
affine connection $\nabla $ on $M$ defines a derivation 
\begin{equation*}
d_{\nabla }^{s}:\Sigma ^{\cdot }\rightarrow \Sigma ^{\cdot +1}
\end{equation*}%
of degree one in this algebra.

Because any such derivation is defined by its action on generators of the
algebra we define $d_{\nabla }^{s}$ as follows:%
\begin{eqnarray*}
d_{\nabla }^{s} &=&d:C^{\infty }\left( M\right) \rightarrow \Omega
^{1}\left( M\right) =\Sigma ^{1}(M), \\
d_{\nabla }^{s} &:&\Omega ^{1}\left( M\right) =\Sigma ^{1}(M)\stackrel{d_{\nabla
}}{\longrightarrow }\Omega ^{1}\left( M\right) \otimes \Omega ^{1}\left(
M\right) \stackrel{\rm{Sym}}{\longrightarrow }\Sigma ^{2}(M).
\end{eqnarray*}

Let choose local coordinate $\left( x_{1},x_{2}\right) $ on $M,$ and let $%
\left( x_{1},x_{2},w_{1},w_{2}\right) $ be the induced local coordinates in
the tangent bundle. Then elements of the symbol algebra $\Sigma ^{\cdot }$
we shall write as polynomials of the form 
\begin{equation*}
\Sigma _{\alpha }a_{\alpha }\left( x_{1},x_{2}\right) w^{\alpha },
\end{equation*}%
where $\alpha =\left( \alpha _{1},\alpha _{2}\right) $ are multi indices and 
$a_{\alpha }\left( x_{1},x_{2}\right) $ are smooth functions.

Then the derivation $d_{\nabla }^{s}$ takes the following form%
\begin{equation*}
d_{\nabla }^{s}=w_{1}\partial _{x_{1}}+w_{2}\partial _{x_{2}}-\Sigma
_{j,k}\Gamma _{jk}^{1}w_{j}w_{k}\partial _{w_{1}}-\Sigma _{j,k}\Gamma
_{jk}^{2}w_{j}w_{k}\partial _{w_{2}}.
\end{equation*}%
Remark that the symbol of the operator (as of any derivation) at a covector
equals to symmetric multiplication in the algebra by the covector.
Therefore, the value of the symbol of $k$-th power $\left( d_{\nabla
}^{s}\right) ^{k}:\Sigma ^{\cdot }\rightarrow \Sigma ^{\cdot +k}$ at a
covector equals to symmetric multiplication by the $k$-th degree of the
covector.

Let's $\sigma \in \Sigma _{k}$ be a symbol. We define a differential
operator $\widehat{\sigma }\in \mathbf{Diff}_{k}\left( \mathbf{1}\right) $
as follows:%
\begin{equation}
\widehat{\sigma }\left( f\right) \overset{\text{def}}{=}\frac{1}{k!}%
\left\langle \sigma ,\left( d_{\nabla }^{s}\right) ^{k}\left( f\right)
\right\rangle .  \label{QuantForm}
\end{equation}%
Then, due to the above remark, the symbol of operator $\widehat{\sigma }$
equals to $\sigma .$

The correspondence%
\begin{eqnarray*}
\mathcal{Q} &:&\Sigma ^{\cdot }\rightarrow \mathbf{Diff}_{\cdot }\left( 
\mathbf{1}\right) =\cup _{k\geq 0}\mathbf{Diff}_{k}\left( \mathbf{1}\right) ,
\\
\mathcal{Q} &:&\left( \sigma _{0},..,\sigma _{k}\right) \longmapsto \Sigma
_{i\geq 0}\widehat{\sigma }_{i},
\end{eqnarray*}%
we call \textit{quantization associated with connection}.

It is worth to note that $\mathcal{Q}$ is the  module isomorphism only but
not the algebra morphism.

Thus, for any operator $A\in \mathbf{Diff}_{k}\left( \mathbf{1}\right) ,$
there is a tensor\\ $\sigma _{A}=\left( \sigma _{0},..,\sigma _{k}\right) ,$
where $\sigma _{i}\in \Sigma _{i},$ such that 
\begin{equation*}
\mathcal{Q}\left( \sigma _{A}\right) =A.
\end{equation*}

We call $\sigma _{A}$ the \textit{total symbol of operator} $A.$ The term $%
\sigma _{k}$ is the\textit{\ principal } or \textit{leading symbol} of the
operator.

The leading symbol $\sigma _{k}$ does not depend on a connection but the
total symbol does.

\subsubsection{Wagner quantization}

We consider here the \textit{Wagner quantization}, i.e. the quantization
associated with the Wagner connection. We'll assume that local coordinates $%
\left( x,y\right) $ are chosen in such a way that 
\begin{equation*}
\sigma _{h}=3\exp (a+b)\left( \exp (b)\partial _{x}+\exp \left( a\right)
\partial _{y}\right) \cdot \partial _{x}\cdot \partial _{y}
\end{equation*}%
in the hyperbolic case, and 
\begin{equation*}
\sigma _{u}=\exp (3r)\left( \sin (h)\partial _{x}+\cos \left( h\right)
\partial _{y}\right) \cdot \left( \partial _{x}^{2}+\partial _{y}^{2}\right)
\end{equation*}%
in the ultrahyperbolic case.

Here $a,b,r,h$ are some smooth functions.

Then in the hyperbolic case the Wagner connection has the following non zero
Christoffel coefficients:%
\begin{equation*}
\Gamma _{11}^{1}=-b_{x},\quad\Gamma _{12}^{1}=-b_{y},\quad\Gamma
_{21}^{2}=-a_{x},\quad\Gamma _{22}^{2}=-a_{y},
\end{equation*}%
and the derivation $d_{\nabla }^{s}$ has the following form%
\begin{equation*}
d_{wh}^{s}=w_{1}\partial _{x}+w_{2}\partial _{y}+\left(
b_{x}w_{1}^{2}+b_{y}w_{1}w_{2}\right) \partial _{w_{1}}+\left(
a_{x}w_{1}w_{2}+a_{y}w_{2}^{2}\right) \partial _{w_{2}}.
\end{equation*}

Then the quantization of the third order symbol 
\begin{equation*}
\sigma _{3}=a_{111}\partial _{1}^{3}+3a_{112}\partial _{1}^{2}\cdot \partial
_{2}+3a_{122}\partial _{1}\cdot \partial _{2}^{2}+a_{222}\partial _{2}^{3}
\end{equation*}%
is the following operator%
\begin{align*}
\widehat{\sigma _{3}}=\,& a_{111}\partial _{1}^{3}+3a_{112}\partial
_{1}^{2}\cdot \partial _{2}+3a_{122}\partial _{1}\cdot \partial
_{2}^{2}+a_{222}\partial _{2}^{3}+ \\
&a_{111}\big( (b_{xx}+2b_{x}^{2})\partial _{x}+3b_{x}\partial
_{x}^{2}\big) + \\
&a_{112}\big( 3b_{y}\partial _{x}^{2}+3( a_{x}+b_{x}) \partial
_{x}\partial _{y}+( 2b_{xy}+3b_{x}b_{y}+a_{x}b_{y}) \partial
_{x}\\
&\phantom{a_{112}}+( a_{xx}+a_{x}^{2}+a_{x}b_{y}) \partial _{y}\big) + \\
&a_{122}\big( 3( a_{y}+b_{y}) \partial _{x}\partial_{y}
+3a_{x}\partial _{y}^{2}+( b_{yy}+b_{y}^{2}+a_{y}b_{y})
\partial _{x}\\
&\phantom{a_{122}}+\left( 2a_{xy}+a_{x}b_{y}+3a_{x}a_{y}\right) \partial
_{y}\big) + \\
&a_{222}\big( 3a_{y}\partial _{y}^{2}+( a_{yy}+a_{y}^{2})
\partial _{y}\big) .
\end{align*}

In particular, for the symbol $\sigma _{h}$ we get%
\begin{align*}
\widehat{\sigma _{h}} =\,&3\exp (a+b)\big( \exp (b)\partial _{x}+\exp (
a) \partial _{y}\big) \partial _{x}\partial _{y}+ \\
&3\exp ( a+2b) \big(b_{y}\partial _{x}^{2}+( a_{x}+b_{y})
\partial _{x}\partial _{y}\big) +\\
&3\exp( 2a+b) \big((
a_{y}+b_{y}) \partial _{x}\partial _{y}+a_{x}\partial _{y}^{2}\big)+ \\
&\exp (a+2b)\big( (2b_{xy}+3b_{x}b_{y}+a_{x}b_{y}) \partial _{x}+( a_{xx}+a_{x}^{2}+a_{x}b_{x}) \partial _{y}\big)\\
&\exp( 2a+b) \big( (b_{yy}+b_{y}^{2}+a_{y}b_{y}) \partial _{x}+( 2a_{xy}+a_{x}b_{y}+3a_{x}a_{y}) \partial _{y}\big).
\end{align*}%

The Wagner quantization of the general second order symbol%
\begin{equation*}
\sigma _{2}=a_{11}\partial _{1}^{2}+2a_{12}\partial _{1}\cdot \partial
_{2}+a_{22}\partial _{2}^{2}
\end{equation*}%
equals%
\begin{eqnarray*}
\widehat{\sigma _{2}} &=&a_{11}\partial _{1}^{2}+2a_{12}\partial
_{1}\partial _{2}+a_{22}\partial _{2}^{2}+ \\
&&a_{11}\left( b_{x}\partial _{x}\right) +a_{12}\left( b_{y}\partial
_{x}+a_{x}\partial _{y}\right) +a_{22}\left( a_{y}\partial _{y}\right) .
\end{eqnarray*}

\subsubsection{Chern quantizations}

Let's apply the quantization procedure using hyperbolic Chern connections (%
\ref{HyperChern}).

Then the derivation $d_{\nabla }^{s}$ has the following form%
\begin{equation*}
d_{ch}^{s}=w_{1}\partial _{x}+w_{2}\partial _{y}-h_{x}w_{1}^{2}\partial
_{w_{1}}+h_{y}w_{2}^{2}\partial _{w_{2}},
\end{equation*}

and the quantization of the third order symbol%
\begin{equation*}
\sigma _{3}=a_{111}\partial _{1}^{3}+3a_{112}\partial _{1}^{2}\cdot \partial
_{2}+3a_{122}\partial _{1}\cdot \partial _{2}^{2}+a_{222}\partial _{2}^{3}
\end{equation*}%
will be the following operator 
\begin{eqnarray*}
\widehat{\sigma _{3}} &=&a_{111}\partial _{1}^{3}+3a_{112}\partial
_{1}^{2}\partial _{2}+3a_{122}\partial _{1}\partial _{2}^{2}+a_{222}\partial
_{2}^{3}+ \\
&&a_{111}\left( 2h_{x}^{2}\partial _{x}-h_{xx}\partial _{x}-3h_{x}\partial
_{x}^{2}\right) -a_{112}\left( h_{xy}\partial _{x}+3h_{x}\partial
_{x}\partial _{y}\right) + \\
&&a_{112}\left( h_{xy}\partial _{y}+3h_{y}\partial _{x}\partial _{y}\right)
+a_{222}\left( 2h_{y}^{2}\partial _{y}+h_{yy}\partial _{y}+3h_{y}\partial
_{y}^{2}\right) .
\end{eqnarray*}%
In particular, for the symbol%
\begin{equation*}
\sigma _{3}=\partial _{1}^{2}\cdot \partial _{2}+\exp (h)\partial _{1}\cdot
\partial _{2}^{2},
\end{equation*}%
we get 
\begin{equation*}
\widehat{\sigma _{3}}=\left( \partial _{1}+\exp (h)\partial _{2}+h_{x}\left(
e^{h}-1\right) \right) \partial _{1}\partial _{2}-\frac{h_{xy}}{3}\partial
_{x}+\frac{h_{xx}}{3}\partial _{y}.
\end{equation*}

Quantization of the second order symbols 
\begin{equation*}
\sigma _{2}=a_{11}\partial _{1}^{2}+2a_{12}\partial _{1}\cdot \partial
_{2}+a_{22}\partial _{2}^{2}
\end{equation*}%
are operators of the following form%
\begin{equation*}
\widehat{\sigma _{2}}=a_{11}\partial _{1}^{2}+2a_{12}\partial _{1}\partial
_{2}+a_{22}\partial _{2}^{2}-a_{11}h_{x}\partial _{x}+a_{22}h_{y}\partial
_{y}.
\end{equation*}

For the case of ultrahyperbolic Chern connection we have the similar
formulae.

The derivation $d_{\nabla }^{s}$ has the form%
\begin{multline*}
d_{cu}^{s}=w_{1}\partial _{x}+w_{2}\partial _{y}\left(
-h_{y}w_{1}^{2}+2h_{x}w_{1}w_{2}+h_{y}w_{2}^{2}\right) \partial
_{w_{1}}\\+\left( -h_{x}w_{1}^{2}-2h_{y}w_{1}w_{2}+h_{x}w_{2}^{2}\right)
\partial _{w_{2}}.
\end{multline*}

and the quantization of the third order symbol is the following%
\begin{align*}
\widehat{\sigma _{3}} =\,&a_{111}\partial _{1}^{3}+3a_{112}\partial
_{1}^{2}\partial _{2}+3a_{122}\partial _{1}\partial _{2}^{2}+a_{222}\partial
_{2}^{3}+ \\
&a_{111}\big( -3h_{y}\partial _{x}^{2}-3h_{x}\partial _{x}\partial
_{y}-( h_{xy}+2h_{x}^{2}-2h_{y}^{2}) \partial _{x}\\
&\phantom{a_{111}}-(h_{xx}-4h_{x}h_{y}) \partial _{y}\big) + \\
&a_{112}\big( 6h_{x}\partial _{x}^{2}-9h_{y}\partial _{x}\partial
_{y}-3h_{x}\partial _{y}^{2}+( 2h_{xx}-12h_{x}h_{y}-h_{yy})
\partial _{x}\\
&\phantom{a_{112}}+( -3h_{xy}-6h_{x}^{2}+6h_{y}^{2}) \partial
_{y}\big) + \\
&a_{122}\big( 6h_{y}\partial _{x}^{2}+9h_{x}\partial _{x}\partial
_{y}-6h_{y}^{2}\partial _{y}^{2}+( 6h_{x}^{2}+3h_{xy}-6h_{y}^{2})
\partial _{x}\\
&\phantom{a_{122}}+( h_{xx}-12h_{x}h_{y}-2h_{yy}) \partial _{y}\big)+ \\
&a_{222}\big( 3h_{y}\partial _{x}\partial _{y}+3h_{x}\partial
_{y}^{2}+( 4h_{x}h_{y}+h_{yy}) \partial _{x}\\
&\phantom{a_{222}}+(2h_{x}^{2}+h_{xy}-2h_{y}^{2}) \partial _{y}\big).
\end{align*}%

The quantization of the second order symbols has the following form%
\begin{eqnarray*}
\widehat{\sigma _{2}} &=&a_{11}\partial _{1}^{2}+2a_{12}\partial
_{1}\partial _{2}+a_{22}\partial _{2}^{2}-a_{11}h_{x}\partial
_{x}+a_{22}h_{y}\partial _{y}+ \\
&&\left( a_{22}-a_{11}\right) \left( h_{y}\partial _{x}+h_{x}\partial
_{y}\right) +2a_{12}\left( h_{x}\partial _{x}-h_{y}\partial _{y}\right) .
\end{eqnarray*}

\section{Differential invariants of differential operators}

\subsection{Jets of differential operators}

Denote by $\pi :\mathbf{Diff}_{3}\left(\mathbf{1}\right) \rightarrow M$ the bundle
of linear scalar differential operators of the third order, and by $\tau
_{i}=S^{i}\tau : S^{i}(TM)\rightarrow M$ the bundles of symmetric
contravariant tensors. Sections of these bundles are linear scalar
differential operators of the third order and symmetric contravariant tensors
(symbols).

If $\left( x,y\right) $ are local coordinates on $M$ , then the induced
local coordinates in these bundles we will denote by $\left( x,y,u^{\alpha
}\right) ,$ where $\alpha =\left( \alpha _{1},\alpha _{2}\right) $ are multi
indexes of length $0\leq \left\vert \alpha \right\vert \leq 3$ - for the
case of bundle $\pi $ and $\left\vert \alpha \right\vert =i$ -for bundles $%
\tau _{i}.$

Thus, for example, if an operator $A$ has the form 
\begin{equation*}
A=\sum_{\alpha ,0\leq \left\vert \alpha \right\vert \leq 3}\frac{a_{\alpha
}\left( x,y\right) }{\alpha !}\partial ^{\alpha },
\end{equation*}%
in local coordinates $\left( x,y\right) ,$ then the corresponding section $%
S_{A}:M\rightarrow \mathbf{Diff}_{3}\left(\mathbf{1}\right) ,$ has the form$\
u^{\alpha }=a_{\alpha }\left( x,y\right) ,$ in the canonical coordinates $%
\left( x,y,u^{\alpha }\right) .$

The similar conversion valids also for contravariant tensors.

Let $\pi _{k}:\mathbf{J}^{k}\left( \pi \right) \rightarrow M$ be the bundles
of $k$-jets of sections of bundle $\pi ,$ i.e. $k$-jets of linear
differential operators of the third order. The induced canonical coordinates
in the bundle $\pi _{k}$ we'll denote by $u_{\beta }^{\alpha },$ where multi
indexes $\beta $ have length $0\leq \left\vert \beta \right\vert \leq k.$ If
we denote by $[A]_{p}^{k}\in \mathbf{J}_{p}^{k}\left( \pi \right) $ the $k$%
-jet of the operator $A$ at a point $p\in M$, then values of $u_{\beta
}^{\alpha }$ at $[A]_{p}^{k}$ equal $\displaystyle\frac{\partial ^{\left\vert \beta
\right\vert }a_{\alpha }}{\partial x^{\beta }}\left( p\right) .$

\subsection{Splitting mapping and invariants}

Denote by $\mathcal{O\subset }\mathbf{J}^{2}\left( \pi \right) ,$ and $%
\mathcal{O}^{\left( k\right) }\subset \mathbf{J}^{k}\left( \pi \right) ,\
k\geq 4,$ the domains of regular $k$-jets , i.e such $k$-jets $[A]_{p}^{k}$
that jets of the symbols $\sigma _{3,A}$ belong to domains of regular jets $%
[\sigma _{3,A}]_{p}^{2}\in \mathcal{O\subset }\mathbf{J}^{2}\left( \tau
_{3}\right) $ and $4$-jets of the symbols belong to the domain of regular
jets of symbols $\mathcal{O}^{\left( 4\right) }\subset \mathbf{J}^{4}\left(
\tau _{3}\right) .$

Consider differential operator 
\begin{equation*}
\nu :\mathcal{O\subset }\mathbf{J}^{2}\left( \pi \right) \rightarrow \tau
_{\leq 3}=\tau _{3}\oplus \tau _{2}\oplus \tau _{1}\oplus \tau _{0},
\end{equation*}%
which sends 2-jet $[A]_{p}^{2}$ to the total symbol $\sigma _{A}=\left(
\sigma _{3,A},\sigma _{2,A},\sigma _{1,A},a_{0}\right) $ with respect to the
Chern connection $\nabla _{\sigma _{3},A}.$

It follows from the construction of the Chern connection that this operator
is natural, i.e. commutes with the action of the diffeomorphism group.

Regularity conditions allow to construct invariant coframe $\left\langle
\theta ,\theta ^{\prime }\right\rangle $ on the space of regular $4$-jets $%
\mathbf{J}^{4}\left( \tau _{3}\right) .$ Let, as in section 4.2., denote by $%
\Xi _{i}$ the universal symbols in $\tau _{i},i=0,1,2,3,$ and let $%
I_{i,\alpha }$ be their components in the coframe $\left\langle \theta
,\theta ^{\prime }\right\rangle .$

By a \textit{natural differential invariant of total symbols} we mean a
function on $\mathbf{J}^{k}\left( \tau _{3}\oplus \tau _{2}\oplus \tau
_{1}\oplus \tau _{0}\right) $ which is rational along fibres of the
projection 
\begin{equation*}
\mathbf{J}^{k}\left( \tau _{3}\oplus \tau _{2}\oplus \tau _{1}\oplus \tau
_{0}\right) \rightarrow M
\end{equation*}

and which is invariant with respect to the action of the diffeomorphism
group.

Then, similar to section 4.3, we get the following result.

\begin{theorem}
The field of natural differential invariants of total symbols is generated
by the basic invariants $I_{i,\alpha }$ and the Tresse derivatives $\delta
_{1},\delta _{2}.$
\end{theorem}

Let's $\nu _{k}:\mathbf{J}^{k+2}\left( \pi \right) \rightarrow \mathbf{J}%
^{k}\left( \tau _{3}\oplus \tau _{2}\oplus \tau _{1}\oplus \tau _{0}\right) $
be the $k$-th prolongations of the differential operator $\nu .$ $\ $Then
these mappings are also natural and the induced morphisms $\nu _{k}^{\ast }$
maps invariants of total symbols to invariants of differential operators.

\begin{theorem}
The field of natural differential invariants of linear scalar differential
operators of the third order is generated by the basic invariants $\nu
_{2}^{\ast }\left( I_{i,\alpha }\right) $ and the Tresse derivatives $\delta
_{1},\delta _{2}.$
\end{theorem}

\subsection{Universal differential operator of the third order}

We apply here the construction of universal differential operator of the
second order (see \cite{LY2}) to the operators of the third order.

We define a total operator%
\begin{equation*}
\square _{3}:C^{\infty }\left( J^{k}\pi \right) \rightarrow C^{\infty
}\left( J^{k+3}\pi \right) ,
\end{equation*}%
by the same universal property%
\begin{equation*}
j_{k+3}\left( A\right) ^{\ast }\left( \square _{3}\left( f\right) \right)
=A\left( j_{k}\left( A\right) ^{\ast }\left( f\right) \right) ,
\end{equation*}%
for all functions $f\in C^{\infty }\left( J^{k}\pi \right) $ and $k=0,1,..,$
and for all operators of the third order $A\in \mathbf{Diff}_{3}\left(
\mathbf{1}\right) .$

The standard arguments show that this operator exists and unique and in
local coordinates has the form%
\begin{equation}
\square _{3}=6\sum_{\alpha ,0\leq \left\vert \alpha \right\vert \leq 3}\frac{%
u^{\alpha }}{\alpha !}\left( \frac{d}{dx}\right) ^{\alpha }.
\label{universe3}
\end{equation}

\begin{theorem}
\begin{enumerate}
\item The universal total differential operator of the third order $\square
_{3}$ exists and unique and has representation (\ref{universe3}) in the
standard local coordinates.

\item The universal total differential operator of the third order $\square
_{3}$ commutes with action of the diffeomorphism group $\mathcal{G}\left(
M\right) $ on the jet bundles $\pi _{k}.$
\end{enumerate}
\end{theorem}

\begin{corollary}
If $f$ is a natural differential invariant of order $k$ for differential
operators of the third order, then $\square _{3}\left( f\right) $ is a
natural invariant of the $\left( k+3\right) $ order.
\end{corollary}

Here, as above, by \textit{natural differential invariant} of order $k$ we
mean a function on $J^{k}\pi $ which is $\mathcal{G}\left( M\right) $%
-invariant and which is rational along fibres of the projection $\pi _{k}.$

We'll also say that two natural differential invariants $I_{1},I_{2}$ are in 
\textit{general position if }%
\begin{equation}
\widehat{d}I_{1}\wedge \widehat{d}I_{2}\neq 0.  \label{general position}
\end{equation}

\begin{theorem}[The principle of n-invariants,(\protect\cite{ALV})]
Let natural differential invariants $I_{1},I_{2}$ be in general position and
let%
\begin{equation*}
J^{\alpha }=\square _{3}\left( I^{\alpha }\right) ,
\end{equation*}%
where $\alpha =\left( \alpha _{1},\alpha _{2}\right) $ and $0\leq \left\vert
\alpha \right\vert \leq 3.$ \newline
Then the field of natural differential invariants is generated by invariants 
$\left( I_{1},I_{2},J^{\alpha }\right) $ and all their Tresse derivatives%
\begin{equation*}
\frac{d^{l}J^{\alpha }}{dI_{1}^{l_{1}}dI_{2}^{l_{2}}},
\end{equation*}%
where $l=l_{1}+l_{2}.$
\end{theorem}

\begin{proof}
First of all remark that condition (\ref{general position}) defines open and
dense domains in the $k$-jet bundles. Take an operator $A\in \mathbf{Diff}%
_{3}\left(\mathbf{1}\right) $ such that locally the section $j_{k}(A)$ belongs to
the domain. Let $f$ \ be a natural invariant of order $k$ and let%
\begin{equation*}
f\left( A\right) =j_{k}(A)^{\ast }\left( f\right)
\end{equation*}%
be its value on the operator $A$.

Then functions $I_{1}\left( A\right) $ and $I_{2}\left( A\right) $ are local
coordinates on $M$ and functions $J^{\alpha }\left( A\right) $ give us a
coefficients of $A$ in these coordinates. By definition, the value $f\left(
A\right) $ depends on these coefficients and their derivatives only, and its
true for almost all operators $A.$ Therefore, function $f$ itself is a
rational function of $J^{\alpha }$ and its Tresse derivatives.
\end{proof}

\section{Differential operators in linear bundles}

Here we extend results of the previous chapter on differential operators
acting in  line bundles $\xi :E\left( \xi \right) \rightarrow M.$

\subsection{Quantization}

Let $\xi :E\left( \xi \right) \rightarrow M$ \ be a line bundle. \ Denote by 
$\Sigma ^{\cdot }\left( \xi \right) =C^{\infty }\left( \xi \right) \otimes
\Sigma ^{\cdot }\left( M\right) $ the graded module of symmetric $\xi $%
-valued differential forms on the surface.

Assume now that both surface $M$ and bundle $\xi $ are equipped with
connections $\nabla ^{M}$ and $\nabla ^{\xi }$ respectively and let $%
d_{\nabla ^{M}}:\Omega ^{1}\left( M\right) \rightarrow \Omega ^{1}\left(
M\right) \otimes \Omega ^{1}\left( M\right) $ and $d_{\nabla ^{\xi
}}:C^{\infty }\left( \xi \right) \rightarrow C^{\infty }\left( \xi \right)
\otimes \Omega ^{1}\left( M\right) $ be their covariant differentials.

Similar to (\cite{LY2}) we define a derivation $d_{\nabla ^{\xi }}^{s}:$ $%
\Sigma ^{\cdot }\left( \xi \right) \rightarrow \Sigma ^{\cdot +1}\left( \xi
\right) $ of degree one over the derivation $d_{\nabla ^{M}}^{s}:$ $\Sigma
^{\cdot }\left( M\right) \rightarrow \Sigma ^{\cdot +1}\left( M\right) $ by
requirement that $d_{\nabla ^{\xi }}^{s}\left( s\otimes \theta \right)
=s\otimes d_{\nabla ^{M}}^{s}\theta +d_{\nabla ^{\xi }}^{s}\left( s\right)
\cdot \theta ,$ where $s\in C^{\infty }\left( \xi \right) ,\theta \in \Sigma
^{\cdot }\left( M\right) .$

It is important to note that both these derivations are differential
operators of the first order and the values of their symbols on a covector $%
\upsilon $ coincide with the symmetric product on the covector.

Let's now $\sigma \in \Sigma _{k}\left( M\right) $ be a symbol of order $k.$
We define its \textit{quantization} $\widehat{\sigma }$ to be a linear
differential operator of order $k$ in the line bundle $\xi ,$ having symbol $%
\sigma $ and acting in the following way%
\begin{equation}
\widehat{\sigma }\left( s\right) =\frac{1}{k!}\left\langle \sigma ,\left(
d_{\nabla ^{\xi }}^{s}\right) ^{k}\left( s\right) \right\rangle .
\label{quant}
\end{equation}%
Let $\left( x_{1},x_{2}\right) $ be local coordinates on $M$ and let $e$ be
a basic section of $\xi $ over the coordinate neighborhood.

Then, as above, we will write down elements of $\Sigma ^{k}\left( \xi
\right) $ in the form 
\begin{equation*}
e\otimes \sum_{\left\vert \alpha \right\vert =k}\alpha _{\alpha }\left(
x\right) w^{\alpha },
\end{equation*}%
where coefficients $a_{\alpha }\left( x\right) $ are smooth functions on the
neighborhood.

Let $\Gamma _{ij}^{l}$ be the Christoffel coefficients of the connection $%
\nabla ^{M}$ and $\theta =\theta _{1}dx_{1}+\theta _{2}dx_{2}$ be the
connection form of the connection $\nabla ^{\xi },$ i.e.%
\begin{equation*}
d_{\nabla ^{\xi }}^{s}\left( e\right) =e\otimes \left( \theta
_{1}w_{1}+\theta _{2}w_{2}\right) .
\end{equation*}

Then the derivation $d_{\nabla ^{\xi }}^{s}$ takes the form%
\begin{equation}
d_{\nabla ^{\xi }}^{s}=w_{1}\left( \partial _{1}+\theta _{1}\right)
+w_{2}\left( \partial _{2}+\theta _{2}\right) -\sum \Gamma
_{ij}^{l}w_{i}w_{j}\partial _{w_{l}}.  \label{Qderivation}
\end{equation}

Remark that formula (\ref{quant}) allows us do define prolongations of
operator $A\in \mathbf{Diff}_{3}\left( \xi \right) $ to operators $A^{\left(
k\right) }\in \mathbf{Diff}_{3}\left( \xi ^{\otimes k}\right) $ acting in
the tensor powers $\xi ^{\otimes k}$ of the bundle, for all $k\in \mathbb{Z}%
\mathit{.}$

\subsection{Connection defined by the 3rd order differential operator}

Let $A\in \mathbf{Diff}_{3}\left( \xi \right) $ be the third order linear
differential operator, acting in the bundle $\xi .$ Denote by $\sigma _{3,A}=%
\rm{smbl}_{3}\left( A\right) \in \Sigma _{3}\left( M\right) $ the
principal symbol of this operator and assume that $\sigma _{3,A}$ is
regular. Let $\nabla =\nabla _{\sigma _{3,A}}$ be the Chern connection
associated with this symbol.

Assume that a connection $\nabla ^{\xi }$ is given, then we define \textit{%
subsymbol }$\sigma _{2,A}\left( \xi \right) \in \Sigma _{2}\left( M\right) $
as follows%
\begin{equation*}
\sigma _{2,A}\left( \xi \right) =\rm{smbl}_{2}\left( A-%
\widehat{\sigma _{3,A}}\right) .
\end{equation*}

\begin{theorem}
There is and unique a connection $\nabla ^{\xi }$ in the linear bundle such
that 
\begin{equation}
\sigma _{2,A}\left( \xi \right) =\lambda \ g_{-1/3}\left( \sigma
_{3,A}\right) ,  \label{coneq}
\end{equation}%
where $g_{-1/3}\left( \sigma _{3,A}\right) $ is the Wagner metric (\ref{gk}%
), associated with symbol $\sigma _{3,A},$ and $\lambda \in C^{\infty
}\left( M\right) $ is a multiplier.
\end{theorem}

\begin{proof}
Assume that operator $A$ has the form 
\begin{eqnarray*}
A &=&a_{1}\partial _{1}^{3}+3a_{2}\partial _{1}^{2}\partial
_{2}+3a_{3}\partial _{1}\partial _{2}^{2}+a_{4}\partial _{2}^{2}+ \\
&&a_{11}\partial _{1}^{2}+2a_{12}\partial _{1}\partial _{2}+a_{22}\partial
_{2}^{2}+b_{1}\partial _{1}+b_{2}\partial _{2}+b_{0}
\end{eqnarray*}%
in a local coordinates $\left( x_{1},x_{2}\right) .$

Then equation (\ref{coneq}) gives us a linear system 
\begin{eqnarray*}
3a_{1}\theta _{1}+3a_{2}\theta _{2}+a_{11} &=&\left(
a_{1}a_{3}-a_{2}^{2}\right) \Delta \left( \sigma _{3,A}\right)
^{-1/3}\lambda , \\
3a_{3}\theta _{1}+3a_{4}\theta _{2}+a_{22} &=&\left(
a_{2}a_{4}-a_{3}^{2}\right) \Delta \left( \sigma _{3,A}\right)
^{-1/3}\lambda , \\
6a_{2}\theta _{1}+6a_{3}\theta _{2}+2a_{12} &=&\left(
a_{1}a_{4}-a_{2}a_{3}\right) \Delta \left( \sigma _{3,A}\right)
^{-1/3}\lambda ,
\end{eqnarray*}%
on coefficients $\left( \theta _{1},\theta _{2}\right) $ of the connection
form and multiplier $\lambda .$

The determinant of this system equals $\Delta \left( \sigma _{3,A}\right) $
and therefore the system has unique solution $\left( \theta _{1},\theta
_{2},\lambda \right) .$
\end{proof}

\begin{remark}
The similar result also valids for the Wagner connection and therefore there
are two natural prolongations of the operators $A\in \mathbf{Diff}_{3}\left(
\xi \right) \ $to the operators $A^{\left( k\right) }\in \mathbf{Diff}%
_{3}\left( \xi ^{\otimes k}\right) .$
\end{remark}

\subsection{Invariants of automorphism group}

By $\pi ^{\xi }:\mathbf{Diff}_{3}\left( \xi \right) \rightarrow M$ denote 
the bundle of linear differential operators, acting in the bundle $\xi $ and
having order 3.

Let%
\begin{equation*}
\pi _{k}^{\xi }:J^{k}\left( \pi ^{\xi }\right) \rightarrow M
\end{equation*}%
be the bundles of k-jets of sections of the bundle $\pi ^{\xi }.$

As above, we denote by $\mathcal{O\subset }J^{2}\left( \pi ^{\xi }\right) $
and $\mathcal{O}^{\left( k\right) }\mathcal{\subset }J^{k}\left( \pi ^{\xi
}\right) ,$ $k\geq 4,$ the domains where the jets of the leading symbol $%
\sigma _{3,A}\in \Sigma _{3},$ $A\in \mathbf{Diff}_{3}\left( \xi \right) ,$
satisfy the same regularity conditions  that in the scalar case, \ $\xi =%
\mathbf{1}$.

As we have seen, differential operators with such regular symbols define:
\begin{enumerate}
\item the Chern or the Wagner connections on $M$ and
\item the affine connection in the line bundle $\xi$.
\end{enumerate}

Therefore, the quantization, defining by these connections, allows us to
split operator $A$ into the sum%
\begin{equation*}
A=\widehat{\sigma _{3,A}}+\widehat{\sigma _{2,A}}+\widehat{\sigma _{1,A}}+%
\widehat{\sigma _{0,A}},
\end{equation*}%
where $\sigma _{i,A}\in \Sigma _{i}.$

Therefore, the operator $A$ is defined by the total symbol $\sigma
_{A}=\oplus _{0\leq i\leq 3}\sigma _{i,A}$ and the connection $\nabla ^{A}$
in the line bundle.

Remark, that all this construction and the splitting is defined by the
second jet of the operator.

In other words, similar to the scalar case, we get a mapping 
\begin{equation*}
\nu ^{\xi }:J^{2}\left( \pi ^{\xi }\right) \supset \mathcal{O\rightarrow }%
\oplus _{0\leq i\leq 3}\tau _{i}\oplus \Lambda ^{2}\tau ^{\ast },
\end{equation*}%
where the component in $\Lambda ^{2}\tau ^{\ast }$ is the curvature form $%
\omega _{A}$ of the connection $\nabla ^{A}.$

Remark that mapping $\nu ^{\xi }$ is invariant of the action of the
automorphism group $\mathbf{Aut}(\xi )$ on the left hand side and the
induced action of the diffeomorphism group on the right hand side. 

Therefore,we have $\mathbf{Aut}(\xi )$-invariants of operators in terms of
invariants of the total symbols and the curvature form.

If we consider now operators with $j_{4}\left( \sigma _{3,A}\right) \in 
\mathcal{O}^{\left( 4\right) }$ and take coordinates of the total symbol and
the curvature form in the invariant coframe $\left\langle \theta ,\theta
^{\prime }\right\rangle $ we get the basic $\mathbf{Aut}(\xi )$-invariants $%
J_{\alpha ,i}^{\xi }$ for the total symbol and $K$ for the curvature form.

Let 
\begin{equation*}
\nu _{4}^{\xi }:J^{4}\left( \pi ^{\xi }\right) \mathcal{\rightarrow }%
J^{2}(\oplus _{0\leq i\leq 3}\tau _{i}\oplus \Lambda ^{2}\tau ^{\ast }),
\end{equation*}%
be the second prolongation of $\nu ^{\xi }.$

Then the following theorem valids.

\begin{theorem}
The field of rational $\mathbf{Aut}(\xi )$-invariants is generated by the
basic invariants%
\begin{equation*}
\left( \nu _{4}^{\xi }\right) ^{\ast }\left( J_{\alpha ,i}^{\xi }\right)
,\left( \nu _{4}^{\xi }\right) ^{\ast }\left( K\right)
\end{equation*}%
and the Tresse derivations $\delta _{1},\delta _{2}.$
\end{theorem}

\section{ Equivalence}

\subsection{Scalar operators}

In this section we consider the equivalence problem for scalar linear
differential operators of the third order, $A\in\mathbf{Diff}_3\left(\mathbf{1}
\right)$.

We say that this operators are \textit{in general position }if for any point 
$a\in M$ \ there are two natural invariants, say $I_{1},I_{2}$, such that their
values $I_{1}\left( A\right) ,I_{2}\left( A\right) $ are independent in a
neighborhood $U$ of the point, i.e. $dI_{1}\left( A\right) \wedge
dI_{2}\left( A\right) \neq 0.$ 

We call functions $I_{1}\left( A\right) ,I_{2}\left( A\right) $ \textit{%
natural coordinates }in $U.$

We also call an atlas $\left\{ U_{\alpha },\phi _{a}:U_{\alpha }\rightarrow 
\mathbf{D}_{\alpha }\subset \mathbb{R}^{2}\right\} $ \textit{natural} if
coordinates $\phi _{a}=\left( I_{1}^{\alpha }\left( A\right) ,I_{2}^{\alpha
}\left( A\right) \right) $ are given by distinct natural invariants: $\left(
I_{1}^{\alpha }\left( A\right) ,I_{2}^{\alpha }\left( A\right) \right) \neq
\left( I_{1}^{\beta }\left( A\right) ,I_{2}^{\beta }\left( A\right) \right)
, $ when $\alpha \neq \beta .$

We denote by $\mathbf{D}_{\alpha \beta }=\phi _{\alpha }\left( U_{\alpha
}\cap U_{\beta }\right) $ and will assume that domains $\mathbf{D}_{\alpha }$
and $\mathbf{D}_{\alpha \beta }$ are connected and simply connected.

Let $A_{\alpha }=\phi _{\alpha \ast }\left( \left. A\right\vert _{U_{\alpha
}}\right) ,$ $A_{\alpha \beta }=\phi _{\alpha \ast }\left( \left.
A\right\vert _{U_{\alpha }\cap U_{\beta }}\right) $ be the images of the
operator $A$ in these coordinates. Then $\phi _{\alpha \beta \ast }\left(
A_{\alpha \beta }\right) =A_{\beta a},$ where $\phi _{\alpha \beta }:\mathbf{%
D}_{\alpha \beta }\rightarrow \mathbf{D}_{\beta \alpha }$ are the transition
mappings.

We call collection $\left( \mathbf{D}_{\alpha },\mathbf{D}_{\alpha \beta
},\phi _{\alpha \beta },A_{\alpha },A_{\alpha \beta }\right) =\mathcal{D}%
\left( A\right) $ \textit{natural model or natural atlas for differential
operator} $A.$

\begin{theorem}
Let $A,A^{\prime }\in \mathbf{Diff}_{3}\left(\mathbf{1}\right) $ be operators in
general position. Then these operators are equivalent with respect to group
of diffeomorphisms $\mathcal{G}\left( M\right) $ if and only if the
following conditions hold:\newline
Open sets 
\begin{equation*}
U_{\alpha }^{\prime }=\left( \phi _{\alpha }^{\prime }\right) ^{-1}\left( 
\mathbf{D}_{\alpha }\right) ,
\end{equation*}%
\newline
where $\phi _{\alpha }^{\prime }=\left( I_{1}^{\alpha }\left( A^{\prime
}\right) ,I_{2}^{\alpha }\left( A^{\prime }\right) \right) :M\rightarrow 
\mathbb{R}^{2}$ constitute a natural atlas for operator $A^{\prime },$ $\phi
_{\alpha \beta }^{\prime }=\phi _{\alpha \beta }:\mathbf{D}_{\alpha \beta
}\rightarrow \mathbf{D}_{\beta \alpha },$ and\newline
\begin{equation*}
A_{\alpha }=\phi _{\alpha \ast }^{\prime }\left( \left. A^{\prime
}\right\vert _{U_{\alpha }^{\prime }}\right) ,A_{\alpha \beta }=\phi
_{\alpha \ast }^{\prime }\left( \left. A^{\prime }\right\vert _{U_{\alpha
}^{\prime }\cap U_{\beta }^{\prime }}\right) .
\end{equation*}
\end{theorem}

\begin{proof}
Indeed, any diffeomorphism $\psi :M\rightarrow M$ such that $\psi _{\ast
}\left( A\right) =A^{\prime }$ transform natural atlas to the natural one
and because of $\psi ^{\ast -1}\left( I\left( A\right) \right) =I\left( \psi
_{\ast }\left( A\right) \right) ,$ for any natural invariant $I,$ this
diffeomorphism has the form of the identity map in the natural coordinates.
\end{proof}

\begin{remark}
In natural coordinates operators $A_{\alpha }$ are defined by 10 functions.
\end{remark}

\subsection{Operators acting in line bundles}

First of all we discuss the problem of lifting diffeomorphisms $\psi
:M\rightarrow M$ \ to automorphisms $\widetilde{\psi }:E\left( \xi \right)
\rightarrow E\left( \xi \right) $ of line bundles.

\begin{lemma}
A diffeomorphism $\psi :M\rightarrow M$ admits a lifting to an automorphism $%
\widetilde{\psi }:E\left( \xi \right) \rightarrow E\left( \xi \right) $ $\ $%
if and only if $\psi ^{\ast }\left( w_{1}\left( \xi \right) \right)
=w_{1}\left( \xi \right) ,$ where $w_{1}\left( \xi \right) \in H^{1}\left( M,%
\mathbb{Z}_{2}\right) $ is the first Stiefel-Whitney class of the bundle.
\end{lemma}

\begin{proof}
Remark that a real linear bundle $\xi $ is trivial if and only if $%
w_{1}\left( \xi \right) \in H^{1}\left( M,\mathbb{Z}_{2}\right) $ vanishes (%
\cite{MS}). 

Therefore, in the case when $w_{1}\left( \xi \right) =0$ the statement of
the lemma trivial.

Let now $w_{1}\left( \xi \right) \neq 0$ and let $\psi ^{\ast }\left( \xi
\right) $ be the line bundle induced by a diffeomorphism $\psi .$ Then, we
have 
\begin{equation*}
w_{1}\left( \xi ^{\ast }\otimes \psi ^{\ast }\left( \xi \right) \right)
=w_{1}\left( \xi ^{\ast }\right) +w_{1}\left( \psi ^{\ast }\left( \xi
\right) \right) =w_{1}\left( \xi \right) +\psi ^{\ast }\left( w_{1}\left(
\xi \right) \right) =0,
\end{equation*}%
if \ $\psi ^{\ast }\left( w_{1}\left( \xi \right) \right) =w_{1}\left( \xi
\right) .$

Therefore, any nowhere vanishing section of the bundle $\xi ^{\ast }\otimes
\psi ^{\ast }\left( \xi \right) $ give us an isomorphism between $\ $bundle $%
\xi $ and $\psi ^{\ast }\left( \xi \right) $ covering the identity map and
then the lift of diffeomorphism $\psi .$
\end{proof}

Similar to the scalar case we say that operator $A\in \mathbf{Diff}%
_{3}\left( \xi \right) $ is \textit{in general position }if for any point $%
a\in M$ \ there are two $\mathbf{Aut}(\xi )$-invariants, say $I_{1}^{\xi
},I_{2}^{\xi },$ such that $dI_{1}^{\xi }\left( A\right) \wedge dI_{2}^{\xi
}\left( A\right) \neq 0$ in some neighborhood $U^{\xi }$ of the point. 

As above we call functions $I_{1}^{\xi }\left( A\right) ,I_{2}^{\xi }\left(
A\right) $ \textit{natural coordinates }in $U,$ and we also call an atlas $%
\left\{ U_{\alpha }^{\xi },\,\phi _{\alpha }^{\xi }:U_{\alpha }^{\xi
}\rightarrow \mathbf{D}_{\alpha }^{\xi }\subset \mathbb{R}^{2}\right\} $ 
\textit{natural} if coordinates $\phi _{\alpha }^{\xi }=\left( I_{\alpha
1}^{\xi }\left( A\right) ,I_{\alpha 2}^{\xi }\left( A\right) \right) $ are
given by distinct natural invariants: $\left( I_{\alpha 1}^{\xi }\left(
A\right) ,I_{\alpha 2}^{\xi }\left( A\right) \right) \neq \left( I_{\beta
1}^{\xi }\left( A\right) ,I_{\beta 2}^{\xi }\left( A\right) \right) ,$ when $%
\alpha \neq \beta .$

We denote by $\mathbf{D}_{\alpha \beta }^{\xi }=\phi _{\alpha }^{\xi }\left(
U_{\alpha }^{\xi }\cap U_{\beta }^{\xi }\right) $ and assume that domains $%
\mathbf{D}_{\alpha }^{\xi }$ and $\mathbf{D}_{\alpha \beta }^{\xi }$ are
connected and simply connected.

Let $\sigma _{A\alpha }=\phi _{\alpha \ast }^{\xi }\left( \left. \sigma
_{A}\right\vert _{U_{\alpha }^{\xi }}\right) ,$ $\sigma _{A\alpha \beta
}=\phi _{\alpha \ast }^{\xi }\left( \left. \sigma _{A}\right\vert
_{U_{\alpha }^{\xi }\cap U_{\beta }^{\xi }}\right) ,\omega _{A\alpha }=\phi
_{\alpha \ast }^{\xi }\left( \omega _{A}\right) $ be the images of the total
symbol $\sigma _{A}$ and the curvature form $\omega _{A}$ in these
coordinates.

Then $\phi _{\alpha \beta \ast }\left( \sigma _{A\alpha \beta }\right)
=\sigma _{A\beta a}\ \ $and $\phi _{\alpha \beta \ast }\left( \omega
_{A\alpha \beta }\right) =\omega _{A\beta a},$ where $\phi _{\alpha \beta }:%
\mathbf{D}_{\alpha \beta }\rightarrow \mathbf{D}_{\beta \alpha }$ are the
transition maps.

We call collection $\left( \mathbf{D}_{\alpha },\mathbf{D}_{\alpha \beta
},\phi _{\alpha \beta },\sigma _{A\alpha },\sigma _{A\alpha \beta },\omega
_{A\alpha },\omega _{A\alpha \beta }\right) =\mathcal{D}\left( A\right) $ 
\textit{natural model or natural atlas for differential operator} $A\in 
\mathbf{Diff}_{3}\left( \xi \right) .$

Let $A^{\prime }\in \mathbf{Diff}_{3}\left( \xi \right) $ be another
operator such that natural model $\mathcal{D}\left( A\right) $ satisfies the
conditions of the above theorem and therefore defines the natural model for
operator $A^{\prime }.$

Denote by $\psi _{\mathbf{D}}:M\rightarrow M$ the diffeomorphism which
equalizes total symbols and curvature forms: 
$(\psi_{\mathbf{D}}) _{\ast }\left( \sigma_{A^{\prime }}\right) =\sigma _{A}$,  $(\psi_{\mathbf{D}}) _{\ast }\left( \omega _{A^{\prime
}}\right) =\omega _{A}$.

Due to the above lemma there is a lifting $\widetilde{\psi_{\mathbf{D}}}\in \mathbf{Aut}%
(\xi )$ of this diffeomorphism, if $\psi _{\mathbf{D}}^{\ast }\left(
w_{1}\left( \xi \right) \right) =w_{1}\left( \xi \right) .$

Let $A^{\prime \prime }=\widetilde{\psi _{\mathbf{D}}}_{\ast }\left(
A^{\prime }\right) $ \ be the image of the operator $A^{\prime }.$ 

Then
\begin{equation*}
\sigma _{A^{\prime \prime }}=\sigma _{A},\quad\omega _{A^{\prime \prime }}=\omega
_{A},
\end{equation*}%
and therefore 
\begin{equation*}
d_{\nabla ^{A^{\prime \prime }}}-d_{\nabla ^{A}}=\theta \otimes \rm{Id},
\end{equation*}%
where $\theta \in \Omega ^{1}\left( M\right) .$

This form is closed $d\theta =0,$ because $\omega _{A^{\prime \prime
}}=\omega _{A},$ and its cohomology class we denote by 
\begin{equation*}
\chi (A,A^{\prime })\in H^{1}\left( M,\mathbb{R}\right) .
\end{equation*}

If this class is trivial, then $\theta =df$ for some function $f\in
C^{\infty }\left( M\right)$, and therefore multiplication by function $\exp
\left( f\right) \in \mathcal{F}\left( M\right)$ establishes the equivalence
between operators $A$ and $A^{\prime \prime}$.

Summarizing, we get the following result.

\begin{theorem}
Let $A,A^{\prime }\in \mathbf{Diff}_{3}\left( \xi \right) $ be operators in
general position. Then these operators are equivalent with respect to group
of automorphisms $\mathbf{Aut}(\xi )$  if and only if the following
conditions hold:

\begin{enumerate}
\item $\phantom{\alpha}$\\
Open sets 
\begin{equation*}
U_{\alpha }^{\xi \prime }=\left( \phi _{\alpha }^{\xi \prime }\right)
^{-1}\left( \mathbf{D}_{\alpha }^{\xi }\right) ,
\end{equation*}%
where $\phi _{\alpha }^{\xi \prime }=\left( I_{\alpha 1}^{\xi }\left(
A^{\prime }\right) ,I_{\alpha 2}^{\xi }\left( A^{\prime }\right) \right)
:M\rightarrow \mathbb{R}^{2},$ constitute a natural atlas for operator $%
A^{\prime },$ $\phi _{\alpha \beta }^{\xi \prime }=\phi _{\alpha \beta
}^{\xi }:\mathbf{D}_{\alpha \beta }^{\xi }\rightarrow \mathbf{D}_{\beta
\alpha }^{\xi },$ and 
\begin{eqnarray*}
\sigma _{A\alpha } &=&\phi _{\alpha \ast }^{\xi \prime }\left( \left. \sigma
_{A^{\prime }}\right\vert _{U_{\alpha }^{\xi \prime }}\right) ,\sigma
_{A\alpha \beta }=\phi _{\alpha \ast }^{\prime }\left( \left. \sigma
_{A^{\prime }}\right\vert _{U_{\alpha }^{\prime }\cap U_{\beta }^{\prime
}}\right) , \\
\omega _{A\alpha } &=&\phi _{\alpha \ast }^{\xi \prime }\left( \left. \omega
_{A^{\prime }}\right\vert _{U_{\alpha }^{\xi \prime }}\right) ,\omega
_{A\alpha \beta }=\phi _{\alpha \ast }^{\prime }\left( \left. \omega
_{A^{\prime }}\right\vert _{U_{\alpha }^{\prime }\cap U_{\beta }^{\prime
}}\right) .
\end{eqnarray*}
\item The obstruction $\chi (A,A^{\prime })\in H^{1}\left( M,\mathbb{R}%
\right) $  is trivial and\\ $\psi _{\mathbf{D}}^{\ast }\left( w_{1}\left( \xi
\right) \right) =w_{1}\left( \xi \right)$.
\end{enumerate}
\end{theorem}


\subsection{Equivalence of differential equations of the 3rd order}

Let $A\in \mathbf{Diff}_{3}\left( \xi \right) $ be an operator with the
regular symbol. Then the operator 
\begin{equation*}
A_{0}=\lambda (A)^{-\frac{3}{2}}A,
\end{equation*}%
where $\lambda \left( A\right) =W_{\sigma _{3,A}}\left( \theta _{A},\theta
_{A}\right)$, $W_{\sigma _{3,A}}$ is the Wagner metric and $\theta _{A}$ is
the first covector in the invariant coframe, we  call a \textit{normalization%
} of operator $A.$

It is easy to see that $W_{f\sigma }=f^{-\frac{2}{3}}W_{\sigma }$ and
therefore the normalization of operator $B=fA,$ $f\in \mathcal{F}\left(
M\right) ,$ and the normalization of operator $A$ related as follows 
\begin{equation*}
B_{0}=\rm{sign}(f)\ A_{0}.
\end{equation*}%
A linear differential equation of 3rd order $\mathcal{E}_{A}\subset
J^{3}\left( \xi \right) $ is defined by a conformal class $\left\{
fA\right\} ,$ where $f\in \mathcal{F}\left( M\right) .$ Using the
normalizations we get the following result.

\begin{theorem}
Linear differential equations of the 3rd order $\mathcal{E}_{A}\subset
J^{3}\left( \xi \right) $ and $\mathcal{E}_{B}\subset J^{3}\left( \xi
\right) $ are equivalent with respect to group of automorphisms $\mathbf{Aut}%
(\xi )$ if and only if the normalization $A_{0}$ is equivalent to
normalization $B_{0}$ or $-B_{0}.$
\end{theorem}
\bigskip

{\bf Acknowledgements}\medskip

This work is supported by the Russian Foundation for Basic Research under grant 18-29-10013 mk.

\end{document}